\documentclass[11pt,reqno]{amsart}

\usepackage{color}
\usepackage{amsmath, amsthm, amssymb}
\usepackage{amsfonts}
\usepackage[ansinew]{inputenc}
\usepackage[dvips]{epsfig}
\usepackage{graphicx}
\usepackage[english]{babel}
\usepackage{enumerate}
\usepackage{hyperref}
\usepackage{mathrsfs}

\theoremstyle{plain}

\newtheorem{thm}{Theorem}[section]
\newtheorem{cor}[thm]{Corollary}
\newtheorem{lem}[thm]{Lemma}
\newtheorem{prop}[thm]{Proposition}

\theoremstyle{definition}
\newtheorem{defi}[thm]{Definition}

\theoremstyle{remark}
\newtheorem{rem}[thm]{Remark}

\numberwithin{equation}{section}

\newcommand{\de}{\partial}

\newcommand{\R}{\mathbb{R}}

\newcommand{\N}{\mathbb{N}}

\newcommand{\eps}{\varepsilon}

\newcommand{\be}{\boldsymbol{e}}

\newcommand{\average}{{\mathchoice {\kern1ex\vcenter{\hrule height.4pt
width 6pt depth0pt} \kern-9.7pt} {\kern1ex\vcenter{\hrule
height.4pt width 4.3pt depth0pt} \kern-7pt} {} {} }}

\def\R{\mathbb{R}}

\DeclareMathOperator*{\argmin}{arg\,min}

\title[Optimal regularity fully nonlinear thin obstacle problem]{Optimal regularity for the fully nonlinear\\ thin obstacle problem}

\author{Maria Colombo}
\address{EPFL SB, Station 8, CH-1015 Lausanne, Switzerland}
\email{maria.colombo@epfl.ch}

\author{Xavier Fern\'andez-Real}
\address{EPFL SB, Station 8, CH-1015 Lausanne, Switzerland}
\email{xavier.fernandez-real@epfl.ch}

\author{Xavier Ros-Oton}
\address{ICREA, Pg. Llu\'is Companys 23, 08010 Barcelona, Spain \& Universitat de Barcelona, Departament de Matem\`atiques i Inform\`atica, Gran Via de les Corts Catalanes 585, 08007 Barcelona, Spain \& Centre de Recerca Matem\`atica, Edifici C, Campus Bellaterra, 08193 Bellaterra, Spain }
\email{xros@icrea.cat}



\thanks{M.C. and X.F. were supported by the Swiss National Science Foundation (SNF grants 200021\_182565 and PZ00P2\_208930) and by the Swiss State Secretariat for Education, Research and lnnovation (SERI) under contract number MB22.00034. X.F. and X.R. were also supported by the AEI project PID2021-125021NA-I00 (Spain). Finally, X.R. has been supported by the European Research Council (ERC) under the Grant Agreement No 801867,  by AGAUR Grant 2021 SGR 00087 (Catalunya), by MINECO Grant RED2022-134784-T (Spain), and by the SpanishAEI through the Mar\'ia de Maeztu Program for Centers and Units of Excellence in R{\&}D CEX2020-001084-M}
\begin{document}

\begin{abstract}
In this work we establish the optimal regularity for solutions to the fully nonlinear thin obstacle problem. In particular, we show the existence of an optimal exponent $\alpha_F$ such that $u$ is $C^{1,\alpha_F}$ on either side of the obstacle. 

In order to do that, we prove the uniqueness of blow-ups at regular points, as well as an expansion for the solution there. 

Finally, we also prove that if the operator is rotationally invariant, then $\alpha_F\ge \frac12$ and the solution is always $C^{1,\frac12}$. 
\end{abstract}

\maketitle 
\section{Introduction}

The aim of this work is to study the optimal regularity and homogeneity of blow-ups at regular points for solutions to the Signorini or thin obstacle problem for fully nonlinear operators.

Monotonicity formulas have been an essential tool in the study of free boundary problems (and minimal surfaces). They are extremely useful to control the solution at increasingly smaller scales and thus allow for a blow-up study to be performed. Nonetheless, monotonicity formulas usually lack flexibility in their application, and as soon as one considers small perturbations of the operators they are often no longer available. This is one of the reasons why, in the last years, the study of free boundary problems for fully nonlinear operators has arisen as a way to study the properties of solutions and contact sets in general contexts where no monotonicity formulas are available; see \cite{Lee98, MS08, Fer16, CRS17, RS17, LP18, DS19, Ind19, SY21, SY21b, WY21}.

In particular, the study of the optimal regularity and homogeneity of blow-ups for the classical Signorini problem is possible thanks to specific monotonicity formulas \cite{AC04, ACS08}. In the context of fully nonlinear operators, however, the lack of monotonicity formulas makes it harder to understand the optimal regularity of solutions.

\subsection{The setting}
Given a domain $D \subset \R^n$, the thin obstacle problem involves a function $u\in C(D)$, an obstacle $\varphi\in C(S)$ defined on a $(n-1)$-dimensional manifold~$S$, a Dirichlet boundary condition given by $g:\de D\to \R$, and an elliptic operator $L$,
\begin{equation}
  \label{eq.classicobst}
  \left\{ \begin{array}{rcll}
  L u&=&0& \textrm{ in } D \setminus \{x\in S : u(x) = \varphi(x)\}\\
  L u&\leq&0& \textrm{ in } D\\
  u&\geq&\varphi& \textrm{ on } S \\
  u&=&g& \textrm{ on }\de D. \\
  \end{array}\right.
\end{equation}

Intuitively, one can think of it as finding the shape of a membrane with prescribed boundary conditions considering that there is a thin obstacle from below.

The first results on the regularity for solutions to thin obstacle problems were established in the seventies by Lewy \cite{Lew68}, Frehse \cite{Fre77}, Caffarelli \cite{Caf79}, and Kinderlehrer \cite{Kin81}. In particular, when $L = \Delta$, Caffarelli showed in 1979 the $C^{1,\alpha}$ regularity of solutions in either side of the obstacle, for some $\alpha > 0$. 

However, while for the (thick) obstacle problem the optimal $C^{1,1}$ regularity of solutions is a simple consequence of maximum principle arguments, for the thin obstacle problem it took 25 years to obtain the optimal H\"older regularity. In 2004, Athanasopoulos and Caffarelli \cite{AC04} showed that solutions to the thin obstacle problem \eqref{eq.classicobst} with $L= \Delta$ are always $C^{1,1/2}$ in either side of the obstacle, and this is optimal. In order to do so, they had to rely on monotonicity formulas specific to the Laplacian, that do not have analogies  to other more general elliptic operators. We refer to \cite{PSU12} or \cite{Fer21} for more details on the thin obstacle problem and its monotonicity formulas.

In this work, we study the fully nonlinear version of problem \eqref{eq.classicobst}. More precisely, we study \eqref{eq.classicobst} with $Lu = F(D^2u)$, a convex fully nonlinear uniformly elliptic operator, and for simplicity we consider $S$ to be flat, and $\varphi = 0$. Since all of our estimates are of local character, we consider the problem in $B_1$,
\begin{equation}
  \label{eq.thinobst}
  \left\{ \begin{array}{rcll}
  F(D^2u)&=&0& \textrm{ in } B_1\setminus \{x_n = 0, u=0\}\\
  F(D^2u)&\leq&0& \textrm{ in } B_1\\
  u&\geq&0& \textrm{ on } B_1\cap\{x_n=0\},\\
  \end{array}\right.
\end{equation}
and we are interested in the study of the regularity of solutions on either side of the obstacle. As recently observed in \cite{FS17, SY21}, this is the problem that appears, for example, when studying the singular points for the (thick) obstacle problem. 

We denote by $\mathcal{M}_n$ the space of matrices $\R^{n\times n}$, and  by $\mathcal{M}^S_n$ the space of symmetric matrices. We assume that $F:\mathcal{M}_n\to \R$ satisfies that
\begin{align}
\label{eq.F}
& F \textrm{ is convex, uniformly elliptic}\\
&\nonumber \textrm{with ellipticity constants } 0 < \lambda \leq \Lambda, \textrm{ and with } F(0)= 0.
\end{align}

\subsection{Known results} A one-sided version of problem \eqref{eq.thinobst} was first studied by Milakis and Silvestre in \cite{MS08}, where the authors proved $C^{1,\eps_\circ}$ regularity of solutions under further assumptions on $F$ and the boundary datum. The second author then showed in \cite{Fer16} the $C^{1,\eps_\circ}$ regularity of solutions for the two-sided problem \eqref{eq.thinobst} with general operators \eqref{eq.F}. In particular, by the results in \cite{Fer16}, a solution $u$ to \eqref{eq.thinobst} satisfies
\[
\|u\|_{C^{1,\eps_\circ}(B_{1/2}^+)}+\|u\|_{C^{1,\eps_\circ}(B_{1/2}^-)}\le C \|u\|_{L^\infty(B_1)},
\]
for some $C$ and $\eps_\circ > 0$ depending only on $n$, $\lambda$, and $\Lambda$, and where we have denoted $B_r^\pm := B_r \cap \{\pm x_n > 0\}$. This is the analogue of the result of Caffarelli \cite{Caf79} for general convex fully nonlinear operators, leaving open the optimal regularity issue.

Later on, the third author and Serra in \cite{RS17} started the study of the contact set $\{x_n = 0, u = 0\}$ for solutions to \eqref{eq.thinobst}. In particular, they established the following dichotomy for any point $x_\circ$ on the free boundary, $x_\circ \in \partial \{u = 0\} \cap B_{1/2} \cap \{x_n= 0\}$ with $|\nabla u(x_\circ)| = 0$: 

There exists $\eps_\circ > 0$ depending only on $\lambda$ and $\Lambda$ such that
\begin{enumerate}
\item ({\it Regular points}) either 
\begin{equation}
\label{eq.reg_points}
c r^{2-\eps_\circ} \le \sup_{B_r(x_\circ)} \, u \le C r^{1+\eps_\circ}
\end{equation}
with $c > 0$, and all $r\in (0,1/4)$,
\item ({\it Degenerate points}) or 
\[
0\le \sup_{B_r(x_\circ)} \, u \le C_\eps r^{2-\eps}
\]
for all $\eps > 0$, and $r \in (0,1/4)$.
\end{enumerate}

This dichotomy is analogous to the case of the Laplacian in \cite{ACS08}, where $\eps_\circ = \frac12$ and free boundary points have frequency either $\mu = 3/2$  or $\mu \ge 2$. 

Moreover, continuing with the results in \cite{RS17},  points satisfying \eqref{eq.reg_points} were defined to be \emph{regular points}, and they are an open subset of the free boundary locally contained in a $C^1$ $(n-2)$-dimensional manifold. 
 More recently, by means of the recent boundary Harnack in \cite{DS20, RT21} one can actually upgrade the regularity of the regular part of the free boundary to $C^{1,\alpha}$ for some $\alpha > 0$.


In order to establish the classification of free boundary points, they proved that if a point is regular, then one can do a partial classification of blow-ups. More precisely, let us suppose that 0 is a regular free boundary point, and let us consider the rescalings 
\[
u_r(x) := \frac{u(rx)}{\|u\|_{L^\infty(B_r)}}.
\]
Then, from the results in \cite{RS17}, $u_r$ converges locally uniformly, up to a subsequence, to some $u_0$ global solution to a two-dimensional fully nonlinear thin obstacle problem, monotone in the direction parallel to the thin space. This information is enough to establish that the free boundary is Lipschitz (and $C^1$) around regular points, without needing a full classification or uniqueness of blow-ups $u_0$.

\subsection{Uniqueness of blow-ups} 
If $u$ is a solution to \eqref{eq.thinobst} and 0 is a regular free boundary point with $|\nabla u(0)| = 0$, then by \cite{RS17} for every sequence $r_k\downarrow 0 $ there is a subsequence $r_{k_j}\downarrow 0 $ such that 
\[
u_{r_{k_j}}(x) \longrightarrow u_0(x)\quad\text{in}\quad C_{\rm loc}^{1,\alpha}(\overline{\R^n_+})\cap C_{\rm loc}^{1,\alpha}(\overline{\R^n_-}),
\]
where $u_0(x) = U_0(x'\cdot e,x_n)$, $U_0\in C^{1,\alpha}(\overline{\R^2_+})\cap C^{1,\alpha}(\overline{\R^2_-})$, and $e\in \mathbb{S}^{n-2}$ is the unit outward normal vector to the contact set in the thin space.  We are denoting $x = (x', x_n)\in \R^{n-1}\times \R$, and $\R^n_\pm := \R^n \cap \{\pm x_n > 0\}$. Moreover (see Lemma~\ref{lem.global2D} below), $U_0$ is monotone in the first coordinate and satisfies a global thin obstacle problem in $\R^2$ with a homogeneous operator $G$,
\begin{equation}
\label{eq.global2D_intro}
\left\{
\begin{array}{rcll}
G(D^2 U_0) &\le& 0& \quad \text{in}\quad  \R^2\\
G(D^2 U_0) &=& 0& \quad \text{in} \quad  \R^2 \setminus (\{x_2 = 0\}\cap \{x_1 \le 0\})\\
U_0 &\ge& 0& \quad \text{on} \quad  \{x_2 = 0\}\\
U_0 &= & 0& \quad \text{on} \quad  \{x_2 = 0\}\cap \{x_1 \le 0\},
\end{array}
\right.
\end{equation}
satisfying
\begin{equation}
\label{eq.global2D_intro2}
U_0(0) = |\nabla U_0(0)| = 0, \qquad |U_0(x)|\le 1+|x|^{2-\eps}\quad\text{in}\quad \R^2.
\end{equation}

In particular, the operator $G$ in \eqref{eq.global2D_intro} is given by $G := F_{(e)} $, where $F_{(e)} :\R^{2\times 2}\to \R$ satisfies \eqref{eq.F} and is defined by 
\begin{equation}
\label{eq.givenby}
F_{(e)}(D^2 w) := F^*(D^2 w(x'\cdot e, x_n)),
\end{equation}
with $F^*$ the recession function or blow-down of $F$, i.e. $F^* = \lim_{t\to 0}tF(t^{-1}\,\cdot\,)$. When $F \neq \Delta$ and more in general when it is not rotationally invariant, the previous operators are expected to substantially depend on the choice of $e$.

The question that was left open in \cite{RS17} is: 
\[
\text{Are all solutions to \eqref{eq.global2D_intro}-\eqref{eq.global2D_intro2} homogeneous? Are they unique?}
\]

This type of question for fully nonlinear operators is usually very difficult. For example, the same question without obstacle in $\R^3$ is a very important open problem. It is known, for instance, that there are singular 2-homogeneous solutions for $n\ge 5$, \cite{NTV12, NV13}, but they do not exist in $\R^3$ or $\R^4$, \cite{NV13b}. Hence, the regularity for fully nonlinear equations is related to understanding whether blow-ups are always homogeneous, which remains an open problem due to the lack of monotonicity formulas. 

Even for Lipschitz and convex $F$ in $\R^2$, the homogeneity and classification of global solutions with sub-cubic growth is open. In this work we tackle these issues, in the context of the thin obstacle problem.

Our first result in this paper is the uniqueness of such blow-ups. Notice that, as a consequence, it follows that blow-ups are homogeneous. For the Signorini problem this is accomplished by means of a monotonicity formula; in our case we need to proceed differently. Notice, also, that in this context the monotonicity in one direction is equivalent to the subquadratic growth in \eqref{eq.global2D_intro2}.

\begin{thm}
\label{thm.uniq_blowup}
Let $G:\mathcal{M}_2\to \R$ be a 1-homogeneous fully nonlinear operator satisfying \eqref{eq.F}. Then, there exists a unique viscosity solution to \eqref{eq.global2D_intro}, $U_0\in C(\R^2)$, such that $\partial_1 U_0\in C(\R^2)$, $\partial_1 U_0\ge 0$, and $|\nabla U_0(0)| = 0$. Moreover, it is homogeneous of degree $1+{\alpha_G}\in (1, 2)$. 
\end{thm}
\begin{rem}
The function $U_0$ divides $\mathbb{S}^1$ into three different sectors (see Figure~\ref{fig.2}), according to its sign. In particular, when $G = \Delta$, $U_0$ corresponds to the third eigenfunction on $\mathbb{S}^1\setminus \{(-1, 0)\}$, which has eigenvalue $\frac32$. For more general operators $G$ it does not make sense to refer to an eigenvalue problem, but still $U_0$ corresponds to the third homogeneity of homogeneous solutions to $G(D^2 u) = 0$ with zero Dirichlet condition on $(-1, 0)\times\{0\}$. 
\end{rem}

Hence, thanks to Theorem~\ref{thm.uniq_blowup}, for a $G$ of the form \eqref{eq.givenby}, we have a unique blow-up  $u_0^{(e)}$ at a regular point with normal to the free boundary $e\in \mathbb{S}^{n-2}$, with homogeneity $1+\alpha_F(e) := 1+\alpha_G = 1+\alpha_{F_{(e)}}$.

\begin{cor}
\label{cor.mainprop}
Let $u$ be a solution to \eqref{eq.thinobst} with $F$ satisfying \eqref{eq.F}, and let 0 be a regular (see Definition~\ref{defi.regular}) free boundary point such that $|\nabla u(0)| = 0$.
Let us denote by $e\in\mathbb{S}^{n-2}$ the unit outward normal to the contact set in the thin space at 0 and by $u_0^{(e)}$ the $(1+\alpha_{F}(e))$-homogeneous solution of Theorem~\ref{thm.uniq_blowup} applied to $F_{(e)}$ in \eqref{eq.givenby}. 

Then, we have
\[
u_r(x) := \frac{u(rx)}{\|u\|_{L^\infty(B_r)}} \longrightarrow u_0^{(e)}(x)\quad\text{in}\quad C^{1,\alpha}_{\rm loc}(\overline{\R^n_+})\cap C^{1,\alpha}_{\rm loc}( \overline{\R^n_-})\quad\text{as}\quad r\downarrow 0.
\]
\end{cor}

Notice that, without loss of generality, after subtracting a linear function if necessary, we can always assume that if 0 is a free boundary point for a solution $u$ to a thin obstacle problem then $|\nabla u(0)| = 0$. 

\subsection{Optimal regularity} 
\label{sec.1homog}


As a consequence of Corollary~\ref{cor.mainprop}, since regular points are those with lowest decay, and $C^{2,\alpha}$ estimates hold outside of the contact set, we can guess the optimal regularity of solutions. 

Let us define, for $\alpha_{F}(e)$ given by Corollary~\ref{cor.mainprop},
\begin{equation}
\label{eq.alphan_opt}
(0,1)\ni {\alpha_F} := \min_{e\in \mathbb{S}^{n-2}} \alpha_{F}(e). 
\end{equation}
For instance, $\alpha_\Delta = \frac12$, since $\alpha_{\Delta}(e) = \frac12$ for all $e\in \mathbb{S}^{n-2}$. The value of ${\alpha_F}$ is the expected optimal regularity of solutions. Observe, though, that this is not an immediate consequence of the classification of blow-ups (contrary to the \emph{almost optimal} regularity in Proposition~\ref{prop.opt} below), and instead will require a very delicate blow-up analysis to obtain an expansion around regular points

In the case of the Laplacian, in \cite{AC04} optimal regularity is obtained without classifying blow-ups first, but by means of monotonicity formulas specific to that problem. Unfortunately, monotonicity formulas are not available for general fully nonlinear operators, so we have to develop other methods.

%
%

This is the main result of this paper, and it reads as follows: 

\begin{thm}
\label{thm.main_1homog}
Let $u$ be a solution to the thin obstacle problem \eqref{eq.thinobst} with $F$ satisfying \eqref{eq.F} and 1-homogeneous.
Let ${\alpha_F}$ be defined by \eqref{eq.alphan_opt} and Corollary~\ref{cor.mainprop}. 

Then, $u\in C^{1,{\alpha_F}}(B_1^+)\cap C^{1,{\alpha_F}}(B_1^-)$ and 
\[
\|u\|_{C^{1,{\alpha_F}}(B_{1/2}^+)}+\|u\|_{C^{1,{\alpha_F}}(B_{1/2}^-)} \le C \|u\|_{L^\infty(B_1)}
\]
where $C$ is a constant depending only on $F$. 

Furthermore, if $0\in \partial\{u = 0\}\cap \{x_n = 0\}$ is a regular free boundary point, and $e\in \mathbb{S}^{n-2}$ is the unit outward normal vector to the contact set at $0$, then we have the expansion 
\begin{equation}
\label{eqn:exp}
u(x) = c_n x_n + c_0 u_0^{(e)}(x) + o\left(|x|^{1+\alpha_{F}(e)+\sigma}\right)
\end{equation}
with $c_n \in \R, c_0 > 0$ and for some $\sigma > 0$ depending only on $F$, where $u_0^{(e)}$ and $\alpha_{F}(e)$ are given by Corollary~\ref{cor.mainprop}. 

In particular, if $|\nabla u(0)|  =0$, the growth in \eqref{eq.reg_points} becomes by \eqref{eqn:exp}, for some $c > 0$,
\[
c r^{1+\alpha_{F}(e)} \le \sup_{B_r(x_\circ)} \, u \le C r^{1+\alpha_{F}(e)} \qquad\mbox{for all } r\in (0,1/4).
\]
\end{thm}

Even in the linear case $F = \Delta$, where one can obtain the optimal regularity through easier arguments, our proof is new and it is the first one that does not use monotonicity formulas. 
%
%
%
%

\subsection{The rotationally invariant case
}
\label{sec.pucci}
In Corollary~\ref{cor.mainprop} we have observed that blow-ups at regular points and their homogeneities depend on the orientation of the free boundary around them. Thus, expansions like the ones in Theorem~\ref{thm.main_1homog} and the corresponding optimal regularity given by ${\alpha_F}$, in practice depend in a fine way on the set of directions that the free boundary can take for a given solution. This type of anisotropy does not appear if one considers, instead, rotationally invariant operators, in which all directions $e\in \mathbb{S}^{n-2}$ will have the same $\alpha_{F} (e)$. 

Therefore, let us consider \emph{Hessian equations} (see, e.g., \cite{Tru95, Tru97}), that can be equivalently defined as those  rotationally invariant or those for which $F(M)$ depends only on the eigenvalues of $M$. 
As a consequence, if $F$ is rotationally invariant, the blow-up $u_0^{(e)}$ and $\alpha_{F}(e)$ as defined in Corollary~\ref{cor.mainprop} are independent of $e$.

In this case,  solutions to fully nonlinear thin obstacle problems with a rotationally invariant operator are always $C^{1,1/2}$, with equality  $\alpha_F=1/2$ if and only if $\{ F \le 0\} = \{{\rm tr}\,(\cdot)\le 0\}$ (that is, the solution satisfies the Signorini problem with the Laplacian as well):

\begin{thm}
\label{thm.rot_inv}
Let $F$ satisfying \eqref{eq.F} be rotationally invariant, and let ${\alpha_F}$ be defined in \eqref{eq.alphan_opt}. Then 
\[
{\alpha_F} \ge \textstyle{\frac12},
\]
with equality if and only if $\{F \le 0 \} = \{{\rm tr}\, (\cdot) \le 0\}$. 

Moreover, if $u$ is a solution to the thin obstacle problem \eqref{eq.thinobst} with operator $F$, then $u \in C^{1,1/2}(B_1^+)\cap C^{1,1/2}(B_1^-)$ and 
\[
\|u\|_{C^{1, 1/2} (B_{1/2}^+)} +\|u\|_{C^{1, 1/2} (B_1^-)}\le C \|u\|_{L^\infty(B_1)}
\]
for some $C$ depending only on $F$. 
\end{thm}


It is unclear whether the property $\alpha_F \ge \frac12$ holds for more general operators as well.

\subsubsection{Pucci operators} In order to establish Theorem~\ref{thm.rot_inv}, and in particular, to get that $\alpha_F\ge\frac12$ for rotationally invariant operators, we need to understand first the fully nonlinear thin obstacle problem for Pucci operators; see Subsection~\ref{ssec.2.1} for the definition and properties.

Pucci operators have the advantage that computations can be made \emph{rather} explicit. When dealing with the thin obstacle problem for Pucci operators, we obtain a more detailed analysis:

\begin{prop}
\label{prop.pucci_res}
Let $u$ be a solution to \eqref{eq.thinobst} with $F = \mathcal{P}^+_{\lambda,\Lambda}$ a Pucci operator. Let $\alpha_{\mathcal{P}^+_{\lambda,\Lambda}}$ be defined as in \eqref{eq.alphan_opt} and Corollary~\ref{cor.mainprop}, the homogeneity of blow-ups at regular points. Then $\alpha_{\mathcal{P}^+_{\lambda,\Lambda}} = \alpha(\omega)$ where $\omega = \Lambda/\lambda$, $\alpha = \alpha(\omega):[1,\infty]\to [\frac12, 1)$ is strictly increasing in $\omega$, and $\alpha_\infty:= \alpha(\infty)$ is the unique solution to 
\[
2\sqrt{\alpha_\infty} = (1-\alpha_\infty)(\pi + 2\arctan\sqrt{\alpha_\infty}),
\]
that is $\alpha_\infty \approx 0.64306995\dots$. 
\end{prop}

\subsection{Non-homogeneous operators} After obtaining the optimal regularity in Theorem~\ref{thm.main_1homog}, the next natural question is to understand what is the role of the 1-homogeneity assumption on the operator. In the next theorem we show that it is crucial: perhaps surprisingly, we obtain that we cannot expect solutions to be  always $C^{1,\alpha_F}$. This behavior is distinctive of nonlinear operators of the form \eqref{eq.F}, and it is not observed, for example, when $F$ is the Laplacian.

\begin{thm}
\label{thm.counterexample}
There exists a solution $u$ to \eqref{eq.thinobst} with some $F$ of the form \eqref{eq.F}, such that $u\notin C^{1,{\alpha_F}}(B_{1/2}^+)$, where $\alpha_F$ is given by \eqref{eq.alphan_opt}. 
\end{thm}
\begin{rem}
In fact, one can build such counterexample in $\R^2$, for any ${\alpha_F} > \frac12$ admissible and for $F$ rotationally invariant.
\end{rem}
Intuitively, the behaviour of the counterexemple in Theorem~\ref{thm.counterexample} arises because blow-downs of $F$ can converge to a limit \emph{very slowly}, slower than any power. 

This slow convergence is essential in order to find a counterexemple: as shown in Theorem~\ref{thm.main_1homog_2}, a modification of the proof of Theorem~\ref{thm.main_1homog} also yields the same result if one assumes that, instead of $F$ being 1-homogeneous, we have 
\[
\|F_t - F^*\| \le C t^{\kappa} \quad\text{as $t \downarrow 0$}, 
\]
for  $F_t = tF(t^{-1}\,\cdot\,)$ and some $\kappa > 0$ (being $\kappa = \infty$ the 1-homogeneous situation, where $F_t = F^* = F$ for all $t>0$). \footnote{Notice that this is the same type of behavior that arises in other situations. For example, for equations in non-divergence form, the rate of convergence of the operator towards its blow-up limit needs to be algebraic (that is, we need H\"older coefficients) in order to gain two derivatives of regularity. Merely continuous coefficients do not imply, in general, that solutions are $C^2$.
}


\medskip

On the other hand, observe that, in general, thanks to Corollary~\ref{cor.mainprop} and by means of a compactness argument, the regularity of the solution is always (almost) the one given by the \emph{worst} homogeneity of the blow-ups at regular points, \eqref{eq.alphan_opt}: 

\begin{prop}
\label{prop.opt} Let $F$ satisfying \eqref{eq.F} and ${\alpha_F}$ be given by \eqref{eq.alphan_opt}.
Let $u$ be a solution to \eqref{eq.thinobst}. Then, $u\in C^{1,{\alpha_F} - \eps}(B_1^+)\cap C^{1,{\alpha_F} - \eps}(B_1^-)$ and 
\[
\|u\|_{C^{1,{\alpha_F} - \eps}(B_{1/2}^+)}+\|u\|_{C^{1,{\alpha_F} - \eps}(B_{1/2}^-)} \le C_\eps \|u\|_{L^\infty(B_1)}
\]
for any $\eps > 0$, for some $C_\eps$ depending only on $\eps$ and $F$. 
\end{prop}
\begin{rem}
Here, the dependence of the constant $C_\eps$ on $F$ is in fact a dependence on the modulus of continuity of blow-downs, $\omega$, $|t F(t^{-1}\cdot) - F^*| = \omega(t)$ when $t\downarrow 0$.
\end{rem}

In particular, thanks to Theorem~\ref{thm.counterexample}, the regularity of solutions obtained by Proposition~\ref{prop.opt} is optimal (in the H\"older class), and can only be improved under further assumptions on the operator (see Sections~\ref{sec.1homog} and \ref{sec.pucci}).

\subsection{Sketches of the proofs} Let us give some of the main ideas in the proofs of our results.

\subsubsection{Uniqueness of blow-ups} 
\label{sssec.uniq}
A recurrent idea in our proofs is the following observation: 

If $w\in C^2$ satisfies a uniformly elliptic equation in non-divergence form in $\R^2$, 
\[
a_{11}(x) \partial_{11} w +a_{12}(x) \partial_{12} w+a_{22}(x) \partial_{22} w = 0,
\]
then a simple computation gives that derivatives of $w$ satisfy an equation in divergence form, 
\begin{equation}
\label{eq.div_intro}
{\rm div}\left(\tilde A(x) \nabla (\partial_1 w)\right) = 0,
\end{equation}
for some uniformly elliptic matrix $\tilde A(x)$ (see  \cite[Section 12.2]{GT83} or \cite[Section~4.2]{FR21}). 

Regular points for the thin obstacle problem collapse all variables in the thin space into one, in the blow-up limit, thus giving a two dimensional solution. The uniqueness in Theorem~\ref{thm.uniq_blowup} then follows by observing that if $u$ and $v$ are both solutions to \eqref{eq.global2D_intro}, then $\alpha u - \beta v$ is a solution to a non-divergence form equation outside of the thin space, for any $\alpha, \beta  >0$ (we are also using that $G$ is 1-homogeneous here). Even if the equation satisfied by $\alpha u - \beta v$ depends on $\alpha$ and $\beta$, it is enough to deduce that $u$ and $v$ satisfy a boundary Harnack inequality and are comparable up to the boundary (see the Appendix in Section~\ref{sec.appendix}). From here, combining it with the interior Harnack, the uniqueness of blow-ups follows by a standard argument. 

The existence (and homogeneity) then follows by constructing an explicit example taking advantage of the recent work on homogeneous solutions in cones, \cite{ASS12}. 

\subsubsection{Optimal regularity} In order to prove the optimal regularity, Theorem~\ref{thm.main_1homog}, we  prove that if $0$ is a regular point with $|\nabla u(0)| = 0$ with blow-up $u_0^{(e)}$, then we can do an expansion 
\[
u(x) = c_0 u_0^{(e)}(x) + o(|x|^{1+\alpha_{F}(e)+\sigma}),
\]
with $\sigma$ depending only on $F$.

This is done by contradiction and a very delicate compactness argument. We assume that the previous does not hold for a sequence $u_k$ with blow-ups $u_{0}^{(k)}$. In particular, we can always subtract the best approximation in $L^2$ of $u_k$ in terms of $u_0^{(k)}$, $u_k - M_k u_0^{(k)}$, and assume by contradiction that it is not lower order. Then, we blow-up as 
\[
w_k := \frac{u_k - M_k u_0^{(k)}}{\|u_k - M_k u_0^{(k)}\|_{L^\infty(B_1)}} \longrightarrow v_\infty,
\]
to obtain first a weak limit $v_\infty$. Observe that $u_k$ and $u_0^{(k)}$ satisfy the same equation outside of the thin space (since $F$ is 1-homogeneous), and thus their difference satisfies a non-divergence form equation. On the other hand, on the thin space they satisfy the same equation \emph{almost} on the same domain (the same up to lower order terms, since we assume 0 is a regular point and the domain is therefore $C^{1,\alpha}$). This is enough for us to prove uniform convergence, since non-divergence form equations have bounded $C^{\beta}$ norm for some $\beta > 0$. Thus, the previous sequence $w_k$ has a strong uniform limit. Moreover, such limit satisfies an equation in non-divergence form outside of a half-space in the thin space (where it is zero), and has a growth rate in $B_R$ bounded by $R^{1+\alpha_F(e) +\sigma}< R^2$ for $R > 1$ and some $\sigma > 0$. 

At this point, we would like to get a contradiction by classifying these solutions. In particular, we would like to show that $v_\infty$ needs to be a non-trivial multiple of $u_0^{(\infty)}$, the limit of the blow-ups $u_0^{(k)}$, which would contradict the fact that we were subtracting the best approximation in $L^2$. In order to do that, we need more information on the limit. 

The first observation is that the limit $v_\infty$ is two dimensional. This holds  because we are assuming that $0$ is a regular point (so the free boundary is smooth around it) so we can use the boundary Harnack for slit domains \cite{DS20, RT21}. 

This is, however, not enough. Even if the operator is the Laplacian, there are global two-dimensional solutions to the Laplace equation, vanishing on a half-line and with strictly subquadratic growth at infinity, that are not  blow-ups at regular points (for example, $r^\frac12 \cos(\frac12 \theta)$ in radial coordinates for the Laplacian). In order to discard this type of solution, we need a stronger ($C^1$) convergence of our sequence $w_k$, at least in the directions parallel to the thin space.

We consider then $\partial_e w_k$, that is, derivatives in the directions perpendicular to the contact set at the origin. From the observation in the proof of the uniqueness of blow-ups above, we can show that the restriction of $\partial_e w_k$ to the plane spanned by the vectors $e$ and $\be_n$ satisfies a problem of the type 
\[
\left\{
\begin{array}{rcll}
{\rm div}(A(x) \nabla w) &=& \partial_{x_1} f(x)&\quad\text{in}\quad B_1\setminus \{x_2 = 0, x_1\le 0\}\\
w &=& 0&\quad\text{on}\quad B_1\cap \{x_2 = 0, x_1\le 0\},
\end{array}
\right.
\]
for some $f$ depending on $w_k$ itself. We then show, on the one hand, that $f\in L^p$ uniformly for some $p > 2$, and on the other hand, that solutions to the previous problem are $C^\delta$ for some $\delta > 0$ --- for this the condition $p > 2$ is crucial. Hence, $w_k$ is converging in $C^1$ in the direction $e$ parallel to the thin space given by the normal to the regular point. The limit $\partial_e v_\infty$ is then continuous (also at 0), and satisfies an equation in divergence form outside of a half-line. By means of a boundary Harnack again, we compare it to $\partial_e u_0^{(\infty)}$ to show that they are equal, up to a non-zero multiple. Hence, we reach a contradiction with the fact that we were subtracting the best multiple along the sequence.

%
%
 
 \subsubsection{Non-homogeneous operators} The almost optimal regularity in Proposition~\ref{prop.opt}, as stated above, follows by a compactness argument after the classification of blow-ups.
 
We construct then the counterexample of Theorem~\ref{thm.counterexample} in two-dimensions, for rotationally invariant operators, and taking advantage of Proposition~\ref{prop.pucci_res}. In particular, we consider a regular point at the origin for a solution to a thin obstacle problem with operator
\[
F(A) = \max_{i\in \N}\{\mathcal{P}^+_{1,\Lambda_i} - c_i\},
\]
where $\Lambda_1 = 1$, $c_1 = 0$, and $\Lambda_i\uparrow 2$, $c_i\uparrow +\infty$. Observe that the constant $c_i$ is fixing the scale at which the solution is behaving like a solution with Pucci operator $\mathcal{P}^+_{1,\Lambda_i}$. Since the homogeneity for the Pucci operators is increasing, the homogeneity of the blow-up limit should be the one for $\mathcal{P}^+_{1,2}$. However, by choosing $c_i$ appropriately we can make sure that the optimal regularity associated to $\mathcal{P}^+_{1,2}$ is not achieved.

\subsection{Structure of the paper} The paper is organized as follows:

In Section~\ref{sec.2} we introduce the notation and preliminaries of our problem. In Section~\ref{sec.3} we prove the existence, uniqueness, and homogeneity of blow-ups, in particular showing Theorem~\ref{thm.uniq_blowup}, Corollary~\ref{cor.mainprop}, and Proposition~\ref{prop.opt}. In Section~\ref{sec.4} we start our proof of the optimal regularity, by showing an expansion of the solution around regular points, which is then used in Section~\ref{sec.5}  to prove our main result, Theorem~\ref{thm.main_1homog}. In Section~\ref{sec.6} we then study rotationally invariant operators, in particular proving Proposition~\ref{prop.pucci_res} first to then show Theorem~\ref{thm.rot_inv}. Finally, in Section~\ref{sec.7} we focus on non-homogeneous operators and construct the counter-example in Theorem~\ref{thm.counterexample}, as well as proving a version of Theorem~\ref{thm.main_1homog} with operators that have a quantified rate of convergence to the recession function, Theorem~\ref{thm.main_1homog_2}.

Section~\ref{sec.appendix} corresponds to an appendix of the work, where we prove a boundary Harnack principle that is used in Section~\ref{sec.3}, Theorem~\ref{thm.boundaryHarnack}.

%
%

\section{Preliminaries} 
\label{sec.2}
\subsection{Notation}\label{ssec.2.1} We denote by $\mathcal{M}_{n, \lambda, \Lambda}$ and $\mathcal{M}^S_{n, \lambda, \Lambda}$ the corresponding spaces consisting of matrices uniformly elliptic with ellipticity constants $0< \lambda \le \Lambda < \infty$. That is, 
\[
A\in \mathcal{M}_{n, \lambda, \Lambda} \quad \Longleftrightarrow \quad A\in \mathcal{M}_{n}\quad \text{and}\quad \lambda {\rm Id} \le A \le \Lambda{\rm Id}.
\]

In particular, we are considering fully nonlinear operators $F:\mathcal{M}_n\to \R$ satisfying \eqref{eq.F} that can be written  (using Einstein's notation) as
\begin{equation}
\label{eq.sup}
F(A) = \sup_{\gamma\in \Gamma} \left( L_\gamma^{ij} A_{ij} + c_\gamma\right) = \sup_{\gamma\in \Gamma} \left( {\rm tr}(L_\gamma A) + c_\gamma\right)
\end{equation}
where $(L_\gamma)_{\gamma\in \Gamma}\subset \mathcal{M}^S_{n, \lambda, \Lambda}$ and $\sup_{\gamma\in \Gamma} c_\gamma = 0$. Being uniformly elliptic implies that, for all $M, N\in \mathcal{M}_n^S$ with $N\ge 0$, then
\begin{equation}
\label{eq.unifellipt}
c_n \lambda \|N\|\le F(M+N)  -F(M) \le c_n^{-1}\Lambda\|N\|
\end{equation}
for some dimensional constant $c_n > 0$ (see \cite{CC95, FR21}). 

On the other hand, for fixed ellipticity constants $\lambda$ and $\Lambda$, Pucci operators are the extremal operators in the class of fully nonlinear uniformly elliptic operators with such ellipticity constants. Namely, we define 
\begin{equation}
\label{eq.puccis}
\begin{split}
\mathcal{P}_{\lambda, \Lambda}^+(M) := \sup_{A \in \mathcal{M}_{n,\lambda, \Lambda}} {\rm tr}\, (AM) = \Lambda \sum_{\mu_i(M) > 0}\mu_i(M) + \lambda \sum_{\mu_i(M) < 0}\mu_i(M) 
\\
\mathcal{P}_{\lambda, \Lambda}^-(M) := \inf_{A \in \mathcal{M}_{n,\lambda, \Lambda}} {\rm tr}\, (AM) = \lambda \sum_{\mu_i(M) > 0}\mu_i(M) + \Lambda \sum_{\mu_i(M) < 0}\mu_i(M) ,
\end{split}
\end{equation}
(see, for example, \cite{CC95} or \cite[Chapter 4]{FR21}), where we have denoted $\mu_i(M)$ the $i$-th eigenvalue of $M$. Observe that only $\mathcal{P}^+_{\lambda,\Lambda}$ satisfies \eqref{eq.F} since it is convex, whereas $\mathcal{P}^-_{\lambda,\Lambda}$ is concave. We call them extremal operators because for any $F$ fully nonlinear operator with ellipticity constants $\lambda\le \Lambda$ we have
\[
\mathcal{P}_{\lambda, \Lambda}^-(M)\le F(M) \le \mathcal{P}_{\lambda, \Lambda}^+(M) \qquad\text{for all}\quad M\in \mathcal{M}_n^S.
\]
We say that $F$ is rotationally invariant (and $F(D^2 u) = 0$ is  a Hessian equation) if 
\[
F(O^T M  O) = F(M) \quad\text{for all $M\in \mathcal{M}^S_n$ and for all $O\in O(n)$}. 
\]

\subsection{Preliminary results} 
Let us assume that $0\in \Gamma(u)$ is a free boundary point such that, after subtracting a hyperplane if necessary, $|\nabla u(0)| = 0$. Let us start defining what it means for 0 to be a regular free boundary point:
\begin{defi}
\label{defi.regular}
Let $u$ solve \eqref{eq.thinobst}, and let us a assume that 0 is a free boundary point such that $|\nabla u(0)| = 0$. We say that 0 is a regular free boundary point if 
\[
\limsup_{r\downarrow 0} \frac{\|u\|_{L^\infty(B_r)}}{r^{2-\eps}} = \infty
\]
for some $\eps > 0$. 
\end{defi}
By the results in \cite{RS17} we know that there exists a constant $\eps_\circ > 0 $ depending only on $\lambda$ and $\Lambda$ such that 
\[
c r^{2-\eps_\circ} \le \sup_{B_r} u \le \frac{1}{c} r^{1+\eps_\circ}
\] 
for some $c > 0$ and for $r$ small enough. In particular, by \cite[Proposition 5.2]{RS17} there exists some global function $u_0\in C^{1,\alpha}_{\rm loc}(\R^2)$ such that for some sequence $r_k\downarrow 0$
\begin{equation}
\label{eq.limit}
u_k(x) = \frac{u(r_k x)}{\|u\|_{L^\infty(B_{r_k})}} \to u_0(x'\cdot e, x_n)\quad\text{locally uniformly},
\end{equation}
where $e\in \mathbb{S}^{n-2}$ and $(e, 0)\in \mathbb{S}^{n-1}$ denotes the outward normal vector to the contact set.

\begin{lem}
\label{lem.global2D}
The blow-up limit $u_0\in C^{1,\alpha}_{\rm loc}(\overline{\R^2_+})\cap C^{1,\alpha}_{\rm loc}(\overline{\R^2_-})$ appearing in \eqref{eq.limit} satisfies $u_0(0) = |\nabla u_0(0)| = 0$ and
\begin{equation}
\label{eq.global2D}
\left\{
\begin{array}{rcll}
G(D^2 u_0) &\le& 0& \quad \text{in}\quad  \R^2\\
G(D^2 u_0) &=& 0& \quad \text{in} \quad  \R^2 \setminus (\{x_2 = 0\}\cap \{x_1 \le 0\})\\
u_0 &\ge& 0& \quad \text{on} \quad  \{x_2 = 0\}\\
u_0 &= & 0& \quad \text{on} \quad  \{x_2 = 0\}\cap \{x_1 \le 0\}, 
\end{array}
\right.
\end{equation}
for some (positively) $1$-homogeneous fully nonlinear operator $G:\mathcal{M}^S_2\to \R$ satisfying \eqref{eq.F} with ellipticity constants $\lambda$ and $\Lambda$, and depending on the normal vector to the free boundary, $e\in \mathbb{S}^{n-2}$. Moreover,
\begin{equation}
\label{eq.monotone_convex}
\partial_{x_1} u_0 \ge 0 \quad\text{in}\quad \R^2,
\end{equation}
and $\partial_{x_1x_1} u_0 \ge 0$; that is, $u_0$ is monotone non-decreasing and convex in the $x_1$ directions. 
\end{lem}
\begin{proof}
Let us denote $v_0(x) = u_0(e\cdot x', x_n)$. By \cite[Theorem 4.1 and Proposition 5.2]{RS17} we know that $u_k$ converges (in $C^1$ norm) to a  global subquadratic solution to a thin obstacle problem that is 2-dimensional and convex in the $x'$-variables, $v_0$, where $e$ denotes the outward normal vector to the contact set in the thin space. It also satisfies $v_0(x', 0) = 0$ in $\{e\cdot x'\le 0\}$. Let us derive the operator of the thin obstacle problem in the limit. 

Each $u_k$ satisfies a thin obstacle problem \eqref{eq.thinobst} in $B_{1/r_k}$ with operator $F_{r_k}$, where\footnote{Observe that we can alternatively write 
\[
F_{r_k}(A) := \frac{r_k^2}{\|u\|_{L^\infty(B_{r_k})}} F\left(\frac{\|u\|_{L^\infty(B_{r_k})}}{r_k^2} A\right) =   \sup_{\gamma\in \Gamma} \left( L_\gamma^{ij} A_{ij} + \frac{r_k^2c_\gamma}{\|u\|_{L^\infty(B_{r_k})}} \right).
\]
}
\[
F_{r_k}(A) := \frac{r_k^2}{\|u\|_{L^\infty(B_{r_k})}} F\left(\frac{\|u\|_{L^\infty(B_{r_k})}}{r_k^2} A\right).
\]

In particular, in the limit $k\to \infty$, since $r_k\downarrow 0$ and by assumption 0 is a regular point (and thus $u$ is strictly superquadratic at the origin) we have that the limit $v_0$ satisfies a global thin obstacle problem, with operator the recession function $F^*$ of $F$
\[
F^*(A) = \lim_{\mu\to\infty} \mu^{-1}F(\mu A).
\]
Observe that $F^*$ is positively 1-homogeneous. 

Moreover, since $u_0$ is two dimensional, from the fact that $v_0$ satisfies a global (in $\R^n$) thin obstacle problem with operator $F^*$, we deduce that $u_0$ satisfies a global (in $\R^2$) thin obstacle problem with operator $F_{(e)}$ given by 
\[
F_{(e)}(D^2 u_0) := F^*(D^2 u_0(x'\cdot e, x_n)). 
\]
It is easy to check that $F_{(e)}:\mathcal{M}_2^S\to \R$ is a 1-homogeneous and convex fully nonlinear uniformly elliptic operator with ellipticity constants $\lambda$ and $\Lambda$. 

The global convexity in \eqref{eq.monotone_convex} holds since $v_0$ is convex in the $x'$-directions, and so $u_0$ is convex in the $x_1$-direction. For the monotonicity, observe that since $u_0 = 0$ in $\{x_2 = 0\}\cap \{x_1\le 0\}$ it immediately holds on $\{x_2 = 0\}$ by convexity. For $\{x_2 \neq 0\}$ it also holds, being $\partial_{x_1} u_0$ the extension of a positive sublinear function that satisfies an equation in bounded measurable coefficients outside of the thin space. (See \cite{RS17}.)
\end{proof}

We will also use the following equivalent characterization of regular points:

\begin{lem}
\label{lem.monotone_regular}
Let $u$ be a solution to \eqref{eq.thinobst}, and let us assume that 0 is a free boundary point with $u(0) = |\nabla u(0)| = 0$. Then, the following are equivalent:
\begin{enumerate}[(i)]
\item $0$ is a regular free boundary point, in the sense of Definition~\ref{defi.regular}.
\item There exists some $\eps > 0$ and $e\in \mathbb{S}^{n-1}\cap \{x_n = 0\}$ such that $\partial_e u \ge 0$ in $B_\eps$ and $\partial_e u \not\equiv 0$. 
\end{enumerate}
\end{lem}
\begin{proof}
The proof of the result is essentially contained in \cite{RS17}, we sketch it for completeness. The fact that (i) implies (ii) appears in the proof of the regularity of the free boundary (see \cite[Proposition 6.1]{RS17}). 

For the reverse implication, we notice that $\partial_e u$ satisfies an equation with bounded measurable coefficients in non-divergence form, in $\{x_n \neq 0\}$. Thus, since it is non-zero, by the Harnack inequality (or strong maximum principle) we have that $\partial_e u \ge c_0 > 0$ in $B_{\eps / 2} \cap\{|x_n|\ge \eps/2\}$. In particular, the same holds for a small cone of directions around $e$ (by $C^{1,\alpha}$ regularity of the solution on either side). Hence, the free boundary in $B_{\eps/2}$ is Lipschitz (proceeding as in \cite[Proposition 6.1]{RS17}). Thus, a standard application of the boundary Harnack inequality in this context (see \cite[Theorem 1.8]{RT21} and the proof of \cite[Corollary 1.9]{RT21}) we have that the free boundary is $C^{1,\alpha}$ in $B_{\eps/2}$.

Once the free boundary is $C^{1,\alpha}$, using a barrier from below (in the same argument as \cite[Proposition 7.1, Lemma 7.2]{RS17}) we have that $\|u\|_{L^\infty(B_r)} \ge cr^{2-\eps_\circ}$ for some $\eps_\circ > 0$ small, and for all $r$ small enough. In particular, it is a regular free boundary point according to Definition~\ref{defi.regular}. 
\end{proof}

The following is a well-known lemma that shows that, in dimension 2, differences of partial derivatives of solutions to a fully nonlinear equation satisfy a divergence form elliptic equation with bounded measurable coefficients. We refer the reader to \cite[Section 12.2]{GT83} or \cite[Section 4.2]{FR21}.

\begin{lem}[\cite{GT83, FR21}]
\label{lem.divergenceform}
Let $u, v\in C^{2,\alpha}(B_1)$ with $B_1\subset \R^2$ such that 
\begin{equation}
\label{eq.solutions}
\tilde F(D^2 u) = \tilde F(D^2 v) = 0\quad\text{in}\quad B_1
\end{equation}
for some $\tilde F$ of the form \eqref{eq.F} and (positively) $1$-homogeneous, with ellipticity constants $0<\tilde \lambda\le \tilde\Lambda < \infty$. Given a fixed $e\in \mathbb{S}^1$, let us define, for any $\alpha, \beta \ge 0$,
\[
w_{\alpha\beta} := \alpha \partial_e u - \beta \partial_e v. 
\]
Then $w_{\alpha\beta}$ satisfies
\begin{equation}
\label{eq.divergenceform}
{\rm div} \left( A_{\alpha\beta}(x)\nabla w_{\alpha\beta}(x)\right)= 0\quad\text{in}\quad B_1,
\end{equation}
in the weak sense, with $A_{\alpha\beta}: B_1\to \mathcal{M}_2$ such that 
\[
\frac{\tilde \lambda}{\tilde\Lambda}{\rm Id} \le A_{\alpha\beta}(x) \le \frac{\tilde \Lambda}{\tilde\lambda}{\rm Id} 
\]
for all $x\in B_1$. 
\end{lem}

\section{Uniqueness of blow-ups at regular points}
\label{sec.3}

In this section we prove the existence of homogeneous blow-ups and their uniqueness, to conclude with the proofs of Theorem~\ref{thm.uniq_blowup}, Corollary~\ref{cor.mainprop}, and Proposition~\ref{prop.opt}. 

\subsection{Existence of homogeneous solutions}
Let us begin by showing that we can always construct a homogeneous solution to our global Signorini problem.

\begin{lem}
\label{lem.existence}
There exists a homogeneous function $v$ satisfying \eqref{eq.global2D}-\eqref{eq.monotone_convex} (up to a multiplicative constant) with homogeneity $1+\alpha_G\in (1, 2)$ depending only on $G$.

Moreover, such solution divides $\mathbb{S}^1$ into three sectors according to its sign, $\Sigma_1$ (negative), $\Sigma_2$ (positive), and $\Sigma_3$ (negative), where the angle of each sector is strictly less than $\pi$. (See Figure~\ref{fig.2}.)
\end{lem}
\begin{figure}
\includegraphics[scale = 1]{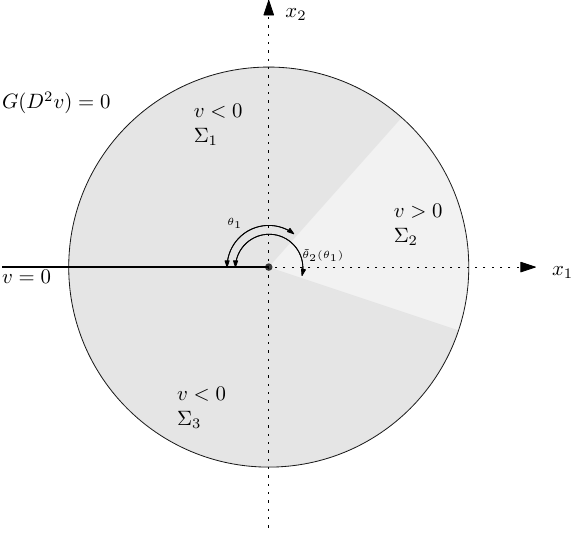}
\caption{The structure of the $(1+{\alpha_G})$-homogeneous solutions to the fully nonlinear thin obstacle problem. The unit circle is divided into three sectors according to the sign of $v$.}
\label{fig.2}
\end{figure}
\begin{proof}
We are going to build a $v$ homogeneous by dividing $\mathbb{S}^1$ into three sectors according to the sign of $v$. 
\\[0.1cm]
{\it Step 1: The three sectors.}
Let us denote by $\Sigma(\phi_0, \phi_1)$ the sector within the angles $ \phi_0<\phi_1$ with respect to $\{x_2 = 0\}\cap \{x_1 \le 0\}$, that is, 
\[
\Sigma(\phi_0, \phi_1) := \{(x_1, x_2) : \phi_0 < \arg (-x_1+ix_2) <\phi_1 \}.
\]

Let us denote by $\theta_i$ the ending angle of the $i$-th sector with respect to $\{x_2 = 0\}\cap \{x_1\le 0\}$, $\Sigma_i$, namely 
\[
\Sigma_i := \Sigma(\theta_{i-1}, \theta_i) 
\]
where we define $\theta_0 = 0$ and we have assumed $\theta_3 = 2\pi$. 

Let us now denote by $\alpha_+(\phi_0, \phi_1) >0$  the homogeneity of the unique function $u$ such that $G(D^2 u) = 0$ in $\Sigma(\phi_0, \phi_1)$, $u = 0$ on $\partial \Sigma(\phi_0, \phi_1)$ and $u > 0 $ in $\Sigma(\phi_0, \phi_1)$. Notice that $\alpha_+(\phi_0, \phi_1)$ is well defined by \cite[Theorem 1.1 and Theorem 1.2]{ASS12} and $\alpha_+(\phi_0, \phi_1)$ is continuous in $\phi_0$ and $\phi_1$ by the stability of solutions to fully nonlinear equations.

Moreover, 
\begin{equation}
\label{eq.mon_angle}
(\phi_0, \phi_1)\subsetneq (\phi_0', \phi_1')\quad\Longrightarrow\quad \alpha_+(\phi_0, \phi_1) > \alpha_+(\phi_0', \phi_1')
\end{equation}
 (see \cite[Section 3]{ASS12}). We similarly define $\alpha_-(\phi_0, \phi_1) > 0$ as  the homogeneity of the unique function $u$ such that $G(-D^2 u) = 0$ in $\Sigma(\phi_0, \phi_1)$, $u = 0$ on $\partial \Sigma(\phi_0, \phi_1)$ and $u > 0 $ in $\Sigma(\phi_0, \phi_1)$.
 
 Observe, also, that 
 \begin{equation}
 \label{eq.halfspace}
 \alpha_+(\phi, \phi+\pi) = \alpha_-(\phi, \phi+\pi) = 1,\end{equation}
 since $G(0) = 0$ and we can choose as the homogeneous solution a hyperplane.
 \\[0.1cm]
 {\it Step 2: Finding the homogeneity.}
 Given $0<\tilde \theta_1<\pi$, choose the unique $\tilde \theta_2 = \tilde \theta_2(\tilde \theta_1)> \pi$ such that $\alpha_-(0, \tilde \theta_1) = \alpha_-(\tilde \theta_2, 2\pi)$. Notice that we can always do that by continuity and monotonicty \eqref{eq.mon_angle}, since $\alpha_-(0, \tilde \theta_1), \alpha_-(\tilde \theta_2, 2\pi) \downarrow 1$ when $\tilde\theta_1 \uparrow \pi$ and $\tilde \theta_2\downarrow \pi$, and $\alpha_-(0, \tilde \theta_1), \alpha_-(\tilde \theta_2, 2\pi) \uparrow \infty$ when $\tilde\theta_1 \downarrow 0$ and $\tilde \theta_2\uparrow 2\pi$. Furthermore, $\tilde \theta_2(\tilde\theta_1)$ is decreasing (and continuous) in $\tilde\theta_1$.  
 
 Now observe that $\alpha_-(0, \tilde \theta_1)$ is decreasing in $\tilde\theta_1$, and that $\alpha_+(\tilde\theta_1, \tilde\theta_2)$ is increasing (again, by \eqref{eq.mon_angle}). By continuity, there is a unique $\tilde\theta_1 = \theta_1$ such that they are equal, which is going to be the one determining our solution (see Figure~\ref{fig.2}). 
 
 Thus, let $\theta_1$ be the unique angle such that 
 \[
1+\alpha_G :=  \alpha_-(0, \theta_1) = \alpha_+(\theta_1, \tilde\theta_2(\theta_1)) = \alpha_-(\tilde\theta_2(\theta_1), 2\pi).
 \]
 
Since we are dividing $\mathbb{S}^1$ into three sectors, at least one of the sectors is (strictly) contained in a half-space, so that $\alpha_G > 0$. Then, from \eqref{eq.halfspace} we get that the other two sectors are also strictly contained in a half-space, so that $\theta_1 < \pi$, $\tilde\theta_2(\theta_1) > \pi$, and $\tilde\theta_2(\theta_1) - \theta_1 < \pi$. Let us now show that $\alpha_G < 1$. 

In order to do that, consider the quadratic function 
\[
q_\xi(x_1, x_2) = x_2(x_1 +\xi x_2)
\]
dividing the plane into four sectors $\Sigma(0, \theta_\xi)$, $\Sigma(\theta_\xi, \pi)$, $\Sigma(\pi, \pi +\theta_\xi)$, and $\Sigma(\pi + \theta_\xi, 2\pi)$, according to the sign of $q_\xi$, where $\theta_\xi$ is such that $\tan(\theta_\xi)  = \xi^{-1}$. Moreover, 
\[
D^2q_\xi(x_1, x_2) = \begin{pmatrix}
0 & 1\\
1 & \xi
\end{pmatrix}.
\]
Then, by continuity and ellipticity \eqref{eq.unifellipt} there is a unique $\bar \xi\in \R$ such that 
\[
G(D^2 q_{\bar \xi}) = 0.
\]

Now, if ${\bar \xi} \ge \theta_1$, then $\Sigma_1 \subseteq \Sigma(0, {\bar \xi})$, which implies (by \eqref{eq.mon_angle}) $1+\alpha_G \ge 2$; and $\Sigma_2\supset \Sigma({\bar \xi}, \pi)$  (since $\tilde\theta_2(\theta_1) > \pi$), which implies $1+\alpha_G < 2$, reaching a contradiction. Hence ${\bar \xi} < \theta_1$ and therefore, again by \eqref{eq.mon_angle}, $1+\alpha_G < 2$, as we wanted to see. 
\\[0.1cm]
{\it Step 3: Conclusion.} We now consider for each sector $\Sigma_1$, $\Sigma_2$, and $\Sigma_3$ the corresponding homogeneous functions described above. Namely, let
\begin{align*}
\varphi_1 : \Sigma_1 \mapsto \R_-,\quad \varphi_2 : \Sigma_2 \mapsto \R_+, \quad \varphi_3 : \Sigma_3 \mapsto \R_-,
\end{align*}
which are $1+\alpha_G$ homogeneous, and we consider them to be defined in $\R^2$ by extending them by zero. Let us also assume that $\|\varphi_i\|_{L^\infty(B_1)} = 1$ for $i=1,2,3$ (since they are defined up to a positive constant). Notice that by Hopf lemma and using that $\varphi_i$ all have the same homogeneity, there exists a unique choice of $c_1, c_3 > 0$ such that 
\[
v = c_1 \varphi_1 + \varphi_2 + c_3\varphi_3
\]
is $C^1$ across the sectors $\Sigma_1$ to $\Sigma_2$, and $\Sigma_2$ to $\Sigma_3$. In particular, since $v$ satisfies $G(D^2 v) = 0$ in each sector, and is $C^1$ across, it satisfies $G(D^2 v) = 0$ in $\R^2 \setminus (\{x_2 = 0\}\cap \{x_1 \le 0\})$ and is the candidate that we were looking for. Finally, since on $\{x_2 = 0\}$ $v$ is monotone increasing, \eqref{eq.monotone_convex} holds (see  \cite{RS17}).
\end{proof}

\subsection{Uniqueness of blow-ups} In fact, the previous solution is the only one, as we show in the following proposition:

\begin{prop}
\label{prop.uniqueness}
There exists a unique function $u_0\in C^1(\overline{\R^2_+})\cap C^1(\overline{\R^2_-})$ satisfying \eqref{eq.global2D}-\eqref{eq.monotone_convex} with $|\nabla u_0(0)| = 0$, up to a multiplicative constant.  
\end{prop}
\begin{proof}
Let us argue by contradiction, and let us suppose that there are two functions $u$ and $v$ satisfying \eqref{eq.global2D}-\eqref{eq.monotone_convex} with $|\nabla u(0)| = |\nabla v(0)| = 0$ (existence being given already by Lemma~\ref{lem.existence}). Observe that we can apply Lemma~\ref{lem.divergenceform} to obtain that for $\alpha, \beta\ge 0$,
\[
w_{\alpha\beta}:= \alpha u_{x_1} - \beta v_{x_1}
\]
all satisfy (outside of the contact set, $\{x_2 = 0\}\cap \{x_1 \le 0\}$) an elliptic equation in divergence form with bounded measurable coefficients in the weak sense, (that is, \eqref{eq.divergenceform}), with ellipticity constants independent of $\alpha$ and $\beta$. In particular, the same is true for both $u_{x_1}$ and $v_{x_1}$. 

Let us assume without loss of generality  (recall $G$ is 1-homogeneous and $u_{x_1}, v_{x_1}\ge 0$), that
\[
u_{x_1}( (1,0)) = v_{x_1}((1,0)) = 1.
\] 

Let us now consider the annulus $B_{3/2}\setminus B_{1/2}$. By the interior Harnack inequality for operators in divergence form (see \cite[Theorem 8.20]{GT83}) applied to both $u_{x_1}$ and $v_{x_1}$, we have that they are also comparable in $(B_{3/2}\setminus B_{1/2})\cap (\{|x_2|\ge \frac18 \}\cup \{x_1 \ge 0\})$, namely, there exists a constant $C$ depending only on $\tilde\lambda$ and $\tilde \Lambda$ such that 
\begin{equation}\label{eq.comparable}
\frac{1}{C}\le \frac{u_{x_1}}{v_{x_1}}\le C
\end{equation}
in $(B_{3/2}\setminus B_{1/2})\cap (\{8|x_2|\ge 1 \}\cup \{x_1 \ge 0\})$.

Now observe that $u_{x_1}$ and $v_{x_1}$ vanish on $\{x_2 = 0\}\cap \{x_1\le 0\}$ and they are non-negative (and continuous) everywhere, so that together with the first observation we can apply the boundary Harnack inequality from Theorem~\ref{thm.boundaryHarnack} in the half balls $B_1((-1, 0))\cap\{\pm x_2\ge 0\}$. In all, we have that $u_{x_1}$ and $v_{x_1}$ are comparable (they satisfy \eqref{eq.comparable}) in the whole annulus $B_{3/2}\setminus B_{1/2}$, for some constant $C$ depending only on $\tilde\lambda$ and $\tilde\Lambda$. In particular,
\[
u_{x_1} - \frac{1}{C} v_{x_1} \ge 0\quad\text{in}\quad B_{3/2}\setminus B_{1/2}.
\]
Since by assumption $u_{x_1} - \frac{1}{C} v_{x_1}$ also satisfies an equation in divergence form and vanishes on $\{x_2  =0\}\cap \{x_1\le 0\}$, we deduce by maximum principle that $u_{x_1} - \frac{1}{C} v_{x_1} \ge 0$ in $B_{3/2}$. The same argument with the other inequality yields that \eqref{eq.comparable} actually holds in the whole ball $B_{3/2}$. 

We now observe that they are in fact comparable in all of $\R^n$. Indeed, up to a rescaling constant $C_R = {u_{x_1}((R, 0))}/{v_{x_1}((R,0))}$ for $R\ge 1$ we can repeat the previous arguments arguments to obtain that $u_{x_1}$ and $C_R v_{x_1}$ are comparable in $B_{3R/2}$, that is, they satisfy \eqref{eq.comparable} for some $C$ that depends only on $\tilde\lambda$ and $\tilde \Lambda$. Using now that $u_{x_1}( (1,0)) = v_{x_1}((1,0)) = 1$  we deduce that $C_R$ is in fact comparable to 1, and $u_{x_1}$ and $v_{x_1}$ are comparable everywhere. That is, \eqref{eq.comparable} holds in $\R^2$. 
  
Let us now define $c_*$ as follows
\[
c_* := \sup\{c \ge 0 : u_{x_1} - c v_{x_1} \ge 0\quad\text{in}\quad \R^2\}. 
\]  
Notice that this set is non-empty, since it contains the value 0, and it is clearly bounded above. Moreover, the function $u_{x_1} - c_* v_{x_1}$ is again a non-negative solution to a divergence form equation \eqref{eq.divergenceform} vanishing on the contact set. In particular, if we define
\[
\kappa := u_{x_1} ((1,0)) - c_* v_{x_1}((1,0))
\]
we have that $\kappa > 0$ by the Harnack inequality, and we can compare the functions 
\[
\frac{1}{\kappa} (u_{x_1} - c_* v_{x_1}) \quad\text{and}\quad v_{x_1}. 
\]
These two functions satisfy the same properties as before, so  repeating the previous procedure we obtain that they are comparable everywhere. In particular, there exists some constant $c > 0$ such that 
\[
\frac{ (u_{x_1} - c_* v_{x_1}) }{\kappa v_{x_1}} \ge c > 0.
\]  
That is, $u_{x_1} - (c_* +\kappa)v_{x_1}\ge 0$ and we get a contradiction with the definition of $c_*$, unless $u_{x_1} = c_*v_{x_1}$ everywhere. In particular, since they coincide at $(1,0)$ we deduce 
\[
u_{x_1} \equiv v_{x_1} \quad\text{in}\quad \R^2. 
\]
That is, $u(x_1, x_2) = v(x_1, x_2) + f(x_2)$. Finally, since both $u$ and $v$ satisfy the same fully nonlinear equation outside the contact set and are $C^{1,\alpha}$, we deduce that $f$ needs to be $C^{1,\alpha}$ and satisfy a bounded measurable coefficient equation in $\R^2$. In particular, it must be a hyperplane. Since $u(0) = v(0) = |\nabla u(0)| = |\nabla v(0)| = 0$ we have that $f_2 \equiv 0$ and $u\equiv v$. That is, there exists at most one solution to \eqref{eq.global2D}-\eqref{eq.monotone_convex} with $|\nabla u_0(0)| = 0$.

The existence of a solution now follows from Lemma~\ref{lem.existence}
\end{proof}

\subsection{Proofs of Theorem \ref{thm.uniq_blowup},  Corollary~\ref{cor.mainprop}, and Proposition~\ref{prop.opt}}

\begin{proof}[Proof of Theorem \ref{thm.uniq_blowup}]
It is an immediate consequence of Lemma~\ref{lem.existence} and Proposition~\ref{prop.uniqueness}.
\end{proof}


\begin{proof}[Proof of Corollary~\ref{cor.mainprop}]
Observe that for any sequence $r_k\downarrow 0$ we can find a subsequence converging to a blow-up. Then, thanks to Proposition~\ref{prop.uniqueness} and Lemma~\ref{lem.existence} such blow-up is unique given by the construction in Lemma~\ref{lem.existence}. 
\end{proof}

Once blow-ups at regular points are classified, by compactness we also obtain the regularity of the solution (losing an arbitrarily small power):

\begin{proof}[Proof of Proposition~\ref{prop.opt}]
%
%
Let us rescale and assume $\|u\|_{L^\infty(B_1)} = 1$. 

We will show that, given $\eps > 0$ fixed, if 0 is a free boundary point, then 
\begin{equation}
\label{eq.toshow_conclude}
\|u\|_{L^\infty(B_r)} \le C_\eps r^{1+{\alpha_F}-\eps}
\end{equation}
for some $C$ depending only on $\eps$ and $F$. 

Indeed, let us suppose it is not true. That is, there exists a sequence $u_k$ of solutions to \eqref{eq.thinobst} with operator $F$, 0 a free boundary point for each $u_k$, and such that 
\[
\theta(r) = \sup_{k\in \N}\,\sup_{\rho > r} \rho^{-1-{\alpha_F}+\eps} \|u_k\|_{L^\infty(B_\rho)}\to \infty \quad\text{as}\quad r\downarrow 0, 
\]
where $\theta(r)$ is a monotone function in $r$. Take $r_m\downarrow 0$ and $k_m\in \N$ sequences such that 
\[
\theta_m := r_m^{-1-{\alpha_F}+\eps}\|u_{k_m}\|_{L^\infty(B_{r_m})} \ge \frac{\theta(r_m)}{2} \to +\infty \quad \text{as}\quad m\to \infty,
\]
and define 
\[
v_m(x) := \frac{u_{k_m}(r_m x)}{\|u_{k_m}\|_{L^\infty(B_{r_m})}}.
\]
Observe now that 
\[
\|v_m\|_{L^\infty(B_1)} =1,\qquad\text{and}\qquad D^2_{x'} v_m \ge -Cq_m^{-1}\quad\text{in}\quad B_{1/(2r_m)}
\]
with $q_m := r_m^2 \|u_{k_m}\|^{-1}_{L^\infty(B_{r_m})} \ge  r_m^{-1+{\alpha_F} -\eps}\theta(r_m) \to +\infty$ as $m\to \infty$, and we are using that solutions to the fully nonlinear thin obstacle problem are semiconvex in the directions parallel to the thin space, \cite[Proposition 2.2]{Fer16}. Moreover, we have that $v_m$ satisfies $F_{m}(D^2 v_m)= 0$ in $\{x_n \neq 0\}$ and $\{x_n = 0\}\cap \{u_{k_m}(r_m\,\cdot\,) > 0\}$, where $F_{m}(A) = q_m^{-1} F(q_m A)$.

In all, $v_m$ satisfy a thin obstacle problem in $B_{1/r_m}$ with operator $F_{m}$ such that $\alpha_{F_{m}} = {\alpha_F}$. Now, since
\[
\|v_m\|_{L^\infty(B_R)} = \frac{\|u_{k_m}\|_{L^\infty(B_{Rr_m})}}{\|u_{k_m}\|_{L^\infty(B_{r_m})}} \le 2R^{1+{\alpha_F}-\eps}\frac{\theta(Rr_m)}{\theta(r_m)}\le 2R^{1+{\alpha_F}-\eps}
\]
for all $1/r_m \ge R\ge 1$, by the regularity estimates for the thin obstacle problem we can let $m\to \infty$ and converge to some global solution $v_\infty$ to the fully nonlinear thin obstacle problem, with the 1-homogeneous operator $F^*$ (the blow-down of $F$), such that 
\[
\|v_\infty\|_{L^\infty(B_R)}\le R^{1+{\alpha_F} -\eps},\quad\text{for all}\quad R\ge 1\qquad\text{and}\qquad D^2_{x'} v_\infty \ge 0\quad\text{in}\quad\R^n. 
\]

However, from the existence and uniqueness of blow-ups at regular free boundary points Lemma~\ref{lem.global2D}, Proposition~\ref{prop.uniqueness} and Lemma~\ref{lem.existence}, $v_\infty$ is either 0, or $\alpha$-homogeneous, with $\alpha \ge 1+{\alpha_F}$. From the growth condition, we obtain $v_\infty\equiv 0$, contradicting $\|v_m\|_{L^\infty(B_1)}= 1$ for all $m\in \N$. 

Hence, \eqref{eq.toshow_conclude} holds. From here, now, obtaining the estimates is standard, combining \eqref{eq.toshow_conclude} with interior $C^{2,\alpha}$ estimates for fully nonlinear convex operators. 
\end{proof}

\section{Expansion around regular points} 
\label{sec.4}
The goal of this section is to prove the following result, saying that the solution has an expansion around regular points.

\begin{thm}
\label{thm.expansion}
 Let $u$ be a solution to the fully nonlinear thin obstacle problem, \eqref{eq.thinobst}. Assume, moreover, that $F$ is 1-homogeneous and of the form \eqref{eq.F}, and that 0 is a regular free boundary point satisfying $|\nabla u(0)| = 0$, and such that $\nu\in\mathbb{S}^{n-1}$ is the unit outward normal to the contact set on the thin space at 0. 
 
 Let $u_0^\nu$ denote the unique blow-up at zero (given by Lemma~\ref{lem.existence}, in the direction $\nu$), with homogeneity $1+\alpha_\nu$. Then, we have the expansion
 \[
 u(x) = c_0 u_0^\nu(x) + o\left(|x|^{1+\alpha_\nu + \sigma}\right)
 \]
 for some $c_0 > 0$ and $\sigma > 0$. The constant $\sigma$ depends only on $F$. 
\end{thm}

We begin by proving the following regularity result for equations in divergence form in slit domains.

\begin{lem}
\label{lem.weird_reg_est}
Let $u\in W^{1,p}(B_1)\cap C^0(B_1)$ for $B_1\subset \R^2$ and $p > 2$, $f\in L^p(B_1)$, and let us assume that $u$ satisfies
\[
\left\{
\begin{array}{rcll}
{\rm div}(\tilde A(x) \nabla u) &=& \partial_{x_1} f(x)&\quad\text{in}\quad B_1\setminus \{x_2 = 0, x_1\le 0\}\\
u &=& 0&\quad\text{on}\quad B_1\cap \{x_2 = 0, x_1\le 0\},
\end{array}
\right.
\]
where $\tilde A(x)$ is uniformly elliptic with ellipticity constants $\lambda$ and $\Lambda$.

Then $u\in C^\delta(B_1)$ for some $\delta > 0$ depending only on $\lambda$, $\Lambda$, and $p$, and 
\[
[u]_{C^\delta(B_{1/2})}\le C \left(\|u\|_{L^\infty(B_1)}+\|f\|_{L^p(B_1)}\right) 
\]
for some $C$ depending only on $\lambda$, $\Lambda$, and $p$.
\end{lem}
\begin{proof}
Let us start by dividing $u$ by $\|u\|_{L^\infty(B_1)} + \frac{1}{\eta_\circ}\|f\|_{L^p(B_1)}$ so that we can assume $\|u\|_{L^\infty(B_1)} \le 1$ and $\|f\|_{L^p(B_1)}\le \eta_\circ$, for some $\eta_\circ$ small enough to be chosen, depending only on $p$, $\lambda$, and $\Lambda$. 

By standard results for divergence-type equations, this type of equation has interior and boundary H\"older regularity estimates (see \cite[Theorem 8.24, Theorem 8.31]{GT83}). Thus, the only problem occurs at the origin. That is, if we show that any such solution $u$ is H\"older continuous at the origin (quantitatively, i.e., by putting a barrier from above and below), then by a standard application of interior and boundary regularity estimates we are done. 

Let us define $v_\pm\in W^{1,2}(B_1)$ to be the unique solution to 
\[
\left\{
\begin{array}{rcll}
{\rm div}(\tilde A(x) \nabla v_\pm ) &=& \partial_{x_1} f(x)&\quad\text{in}\quad B_1\setminus \{x_2 = 0, x_1\le 0\}\\
v_\pm  &=& 0&\quad\text{on}\quad B_1\cap \{x_2 = 0, x_1\le 0\}\\
v_\pm  &=& \pm 1&\quad\text{on}\quad \partial B_1.
\end{array}
\right.
\]

In particular, by maximum principle $v_+ \ge u$ in $B_1$. Let us show that $v_+$ is H\"older continuous, quantitatively, at the origin. 

Let us write $v_+ = v_1 + v_2$, where $v_1, v_2\in W^{1,2}(B_1)$ satisfy 
\[
\left\{
\begin{array}{rcll}
{\rm div}(\tilde A(x) \nabla v_1) &=& 0&\quad\text{in}\quad B_1\setminus \{x_2 = 0, x_1\le 0\}\\
v_1 &=& 0&\quad\text{on}\quad B_1\cap \{x_2 = 0, x_1\le 0\}\\
v_1 &=& 1&\quad\text{on}\quad \partial B_1.
\end{array}
\right.
\]
and
\[
\left\{
\begin{array}{rcll}
{\rm div}(\tilde A(x) \nabla v_2) &=& \partial_{x_1} f(x)&\quad\text{in}\quad B_1\setminus \{x_2 = 0, x_1\le 0\}\\
v_2 &=& 0&\quad\text{on}\quad B_1\cap \{x_2 = 0, x_1\le 0\}\\
v_2 &=& 0&\quad\text{on}\quad \partial B_1.
\end{array}
\right.
\]

By maximum principle (for example, \cite[Theorem 8.15]{GT83}), since $p > 2$ and we are in $\R^2$, we have that
\[
\|v_2\|_{L^\infty(B_1)}\le C \|f\|_{L^p(B_1)} \le C \eta_\circ
\]
for some $C$ depending only on $p$, $\lambda$ and $\Lambda$.

On the other hand, by boundary regularity estimates we know that $v_1$ (which is nonnegative) is close to zero near $(-\frac12, 0)$. Combined with a sequence of applications of the interior Harnack inequality (for the function $1-v_1$) along balls covering $\partial B_{\frac12}$ we get that $v_1$ is bounded on $\partial B_{1/2}$ by $1-\theta_1$, for some $\theta_1$ small enough depending only on $\lambda$ and $\Lambda$. Hence, by maximum principle 
\[
\|v_1\|_{L^\infty(B_{1/2})}\le 1-\theta_1. 
\]

In all, we have that if $\eta_\circ$ is small enough depending only on $p$, $\lambda$, and $\Lambda$, then 
\[
\|v_+\|_{L^\infty(B_{1/2})}\le 1-\tilde \theta,
\]
for some $\tilde\theta > 0$ depending only on $p$, $\lambda$, and $\Lambda$. We can now iterate the process and repeat the argument, since $p > 2$ and assuming $\tilde\theta$ smaller if necessary depending on $p$, to get that 
\[
\|v_+\|_{L^\infty(B_{r})}\le Cr^{\sigma}
\]
for some $C$ and $\sigma > 0$ depending only on $p$, $\lambda$, and $\Lambda$. Repeating the process from below, we obtain the same result for $v_-$, thus having constructed barriers from above and below for $u$. By means of interior and boundary regularity estimates, this concludes the proof. 
\end{proof}

The following follows as \cite[Lemma 5.3]{RS16} and will be useful below. 

\begin{lem}
\label{lem.aux}
Let $u_0$ with $\|u_0\|_{L^\infty(B_1)} = 1$ be  of the form $u_0(x) = v(x'\cdot \be_1, x_n)$, where $v$ is one of the blow-ups constructed in Lemma~\ref{lem.existence} coming from some operator $G$ satisfying \eqref{eq.F}, $\tilde \alpha$-homogeneous with $\tilde\alpha\in (1, 2)$, and let $u\in C(B_1)$ with $\|u\|_{L^\infty(B_1)} = 1$. Let us define 
\[
\phi_r(x) := L_*(r) x_n + Q_*(r) u_0(x),
\]
where 
\begin{align*}
(L_*(r), Q_*(r)) & := \argmin_{(L, Q)\in \R^2}\int_{B_r} (u(x) - Lx_n - Qu_0(x))^2 \, dx.
\end{align*}
Assume that, for all $r\in (0, 1)$, we have 
\[
\|u-\phi_r\|_{L^\infty(B_r)}\le C_0 r^{\beta}
\]
with $\beta > \tilde \alpha > 1$ and $C_0 \ge 1$. Then, there is $L, Q\in \R$ with $L, Q \le C C_0$ such that 
\[
\|u - Lx_n - Q u_0\|_{L^\infty(B_r)} \le C C_0 r^\beta\quad\text{for all}\quad r\in (0, 1),
\]
for some constant $C$ depending only on $\beta$, $\tilde \alpha$, and $G$. 
\end{lem}
\begin{proof}
Since both $u_0$ and $x_n$ are two-dimensional functions, let us assume that we are in $\R^2$ (the same argument is valid in $\R^n$). We consider polar coordinates and let us write $u_0(x) = \rho^{\tilde\alpha} \tilde u_0(\theta)$ and $x_n = \rho \sin{\theta}$, for some $\tilde u_0:\mathbb{S}^1 \to \R$ coming from Lemma~\ref{lem.existence}. 

Observe now that, 
\[
|Q_*(2r) - Q_*(r)|r^{\tilde\alpha} \le C\|\phi_{2r}- \phi_r\|_{L^\infty(B_r)} 
\]
since for $\theta = \pi$, $\phi_{2r}(x)-\phi_r(x) = \rho^{\tilde\alpha} (Q_*(2r)-Q_*(r)) \tilde u_0(\pi)$ (where $\rho = |x|$), and $\tilde u_0(\pi) > 0$ by construction (again, see Lemma~\ref{lem.existence}). On the other hand, we also have 
\[
|L_*(2r) - L_*(r)|r \le C\|\phi_{2r}- \phi_r\|_{L^\infty(B_r)} 
\]
since there exists some $\tilde\theta\in (0, \pi)$ such that $\tilde u_0(\tilde\theta) = 0$. In particular, for such angle (that depends on $G$) we have $\phi_{2r}(x)-\phi_r(x) = \rho (L_*(2r)-L_*(r)) \sin (\tilde\theta)$. In all, we have shown that
\begin{equation}
\label{eq.usingthat}
\max\{|Q_*(2r)-Q_*(r)|r^{\tilde \alpha}, |L_*(2r)-L_*(r)|r\} \le C\|\phi_{2r}- \phi_r\|_{L^\infty(B_r)} \le C C_0 r^\beta,
\end{equation}
for some $C$ that depends on $G$

Now the proof is exactly the same as \cite[Lemma 5.3]{RS16}, using \eqref{eq.usingthat}.
\end{proof}

\begin{prop}
\label{prop.opt_reg}
Let $u$ be a solution to \eqref{eq.thinobst} with $F$ 1-homogeneous and of the form \eqref{eq.F}. Let us suppose $\|u\|_{L^\infty(B_1)} = 1$, that 0 is a regular free boundary point with $|\nabla u(0)| = 0$, and that if $u_0$ is the unique blow-up  \eqref{eq.limit} at 0 with homogeneity $1+\alpha_0$ (as constructed in Lemma~\ref{lem.existence}, then 
\begin{equation}
\label{eq.closeness}
\left\|\frac{u(rx)}{\|u\|_{L^\infty(B_r)}} - u_0(x)\right\|_{L^\infty(B_1)}\le \eta\quad\text{for all}\quad r < 1.
\end{equation}

Let us assume, moreover, that $\{u = 0\}\cap \{x_n = 0\}\cap B_1 =: \Omega'$ is a $C^{1,\alpha}$ domain in $\{x_n = 0\}$, with $C^{1,\alpha}$ bounded by $\eta > 0$. 

Then, there exists a constant $Q> 0$ with $Q \le  C$ and $\eta_\circ> 0$ depending only on $F$ and $\alpha$, such that if $\eta < \eta_\circ$ then
\[
|u(x) - Q u_0(x)| \le C |x|^{1+\alpha_0 + \sigma}\quad\text{in}\quad B_1
\]
for some constants $C$ and $\sigma > 0$ depending only on $F$, and $\alpha$.
\end{prop}
\begin{proof}
We will show that there exist constants $Q > 0$ and $L$ with $Q, |L| \le C$ and $\eta_\circ> 0$ such that if $\eta < \eta_\circ$ then
\[
|u(x) - Lx_n - Q u_0(x)| \le C |x|^{1+\alpha_0 + \sigma}\quad\text{in}\quad B_1
\]
for some $\sigma > 0$ and $C$.  In particular, a posteriori, using that $\nabla u(0) = 0$ we deduce $L = 0$. 

Notice also that, if such $Q$ exists, we necessarily have that $Q\ge 0$. Moreover, arguing as in the  proof of the claim \eqref{eq.limsup} we actually have that $Q > 0$. 

We divide the proof into 7 steps. 
\\[0.1cm]
{\it Step 1: The setting.}
Let us argue by contradiction, and let us suppose that there are sequences $\Omega_k$, $u_k$, $\nu_k\in \mathbb{S}^{n-2}$ such that 
\begin{itemize}
\item $u_k$ is a solution to \eqref{eq.thinobst} with operator $F$, such that 0 is a regular free boundary point, $|\nabla u_k(0) | = 0$ and $\|u_k\|_{L^\infty(B_1)} = 1$,
\item the blow-up at 0 of $u_k$ is $u_0^{(k)}$ with homogeneity $1+\alpha_k\in (1, 2)$, and \eqref{eq.closeness} holds for $u_k$ and $u_0^{(k)}$, and for some $\eta$ to be chosen, 
\item the set $\{u_k = 0\}\cap \{x_n = 0\}\cap B_1 = \Omega'_k$ is a $C^{1,\alpha}$ domain in $\{x_n = 0\}$ with $C^{1,\alpha}$ norm bounded by $\eta$,
\item the outward normal to the contact set $\Omega_k'$ at 0 is given by $\nu_k$, so that in particular $u_0^{(k)}(x', x_n) = \tilde u_0^{(k)}(x'\cdot \nu_k, x_n)$ for some $\tilde u_0^{(k)} :\R^2 \to \R$,
\end{itemize}
but they are such that for all $C$ there exists some $k\in \N$ such that there is no constants $L$ and $Q$ satisfying 
\[
|u_k(x) - L x_n - Q u_0^{(k)}(x)|\le C |x|^{1+\alpha_k+\sigma}\quad\text{in}\quad B_1.
\]
In particular, there are no $L$ and $Q$ such that 
\[
|u_k(x) - L x_n - Q u_0^{(k)}(x)|\le C|x|^{1+\Pi} \le C |x|^{1+\alpha_k+\sigma}\quad\text{in}\quad B_1.
\]
where $\Pi := 1+\alpha_\infty+2\sigma$, 
 we are assuming that, up to taking a subsequence, $\alpha_k \to \alpha_\infty \in (0, 1)$, and we are taking $k$ large enough. The constant $\sigma$ will be chosen later in terms of $F$, and $\alpha$.
\\[0.1cm]
{\it Step 2 General properties.} Let us denote 
\[
\phi_{k, r}(x) := L_k(r) x_n + Q_k(r) u_0^{(k)}(x),
\]
where 
\begin{align*}
(L_k(r), Q_k(r)) & := \argmin_{(L,Q)\in \R^2}\int_{B_r} (u_k(x) - Lx_n -Qu_0^{(k)}(x))^2 \, dx.
\end{align*}

In particular, a simple computation yields that 
\[
L_k(r) =  \frac{\langle u_0^{(k)}, u_0^{(k)}\rangle_r \langle x_n, u_k\rangle_r -\langle u_0^{(k)}, u_k\rangle_r \langle u_0^{(k)}, x_n\rangle_r }{\langle u_0^{(k)}, u_0^{(k)} \rangle_r \langle x_n, x_n \rangle_r -\langle u_0^{(k)}, x_n\rangle_r^2}
\]
and
\[
Q_k(r) =  \frac{\langle u_0^{(k)}, u_k\rangle_r \langle x_n, x_n\rangle_r -\langle u_k, x_n\rangle_r \langle u_0^{(k)}, x_n\rangle_r }{\langle u_0^{(k)}, u_0^{(k)} \rangle_r \langle x_n, x_n \rangle_r -\langle u_0^{(k)}, x_n\rangle_r^2},
\]
where for the sake of readability we have denoted 
\[
\langle f, g\rangle_r := \int_{B_r} f(x)g(x)\, dx. 
\]

Since $u_0^{(k)}$ and $x_n$ are  linearly independent, and they are homogeneous, we can assume that there exists some small constant $c_0 > 0$ depending only on $F$ such that 
\begin{equation}
\label{eq.diff}
 \langle x_n, x_n \rangle_r \langle u_0^{(k)}, u_0^{(k)} \rangle_r -\langle u_0^{(k)}, x_n\rangle_r^2 \ge c_0 r^{2(2+\alpha_k +n)} > 0.
\end{equation}

Notice that $Q_k(r) > 0$ if and only if
\[
\int_{B_1} x_n^2\, dx \int_{B_1} u_0^{(k)}(x) \frac{u_k(rx)}{\|u_k\|_{L^\infty(B_r)}}\, dx - \int_{B_1}  \frac{u_k(rx)}{\|u_k\|_{L^\infty(B_r)}} x_n \, dx \int_{B_1} u_0^{(k)}(x) x_n\, dx>0.
\]
From \eqref{eq.closeness} and \eqref{eq.diff} with $r  =1$, we deduce that the previous holds if $\eta$ is small enough (depending on $F$), and thus we get that 
\begin{equation}
\label{eq.Qpositive}
Q_k(r) > 0.
\end{equation}

\noindent {\it Step 3: The blow-up sequence.}
On the other hand, from Lemma~\ref{lem.aux}, we have 
\[
\sup_{k\in \N} \sup_{r > 0}\left\{r^{-\Pi}\|u_k- \phi_{k, r}\|_{L^\infty(B_r)}\right\} = +\infty.
\]

We now claim that there exist subsequences $k_m\to \infty$ and $r_m\downarrow 0$ such that 
\begin{equation}
\label{eq.boundu}
\theta_m := r_m^{-\Pi} \|u_{k_m} - \phi_{k_m, r_m}\|_{L^\infty(B_{r_m})} \to +\infty,
\end{equation}
and if we denote $\phi_m := \phi_{k_m, r_m}$, and 
\[
v_m(x) := \frac{u_{k_m}(r_m x) - \phi_m(r_m x)}{\|u_{k_m} - \phi_m\|_{L^\infty(B_{r_m})}},
\]
then we  have a bound on the growth control of $v_m$ given by
\begin{equation}
\label{eq.growthctrl}
\|v_m\|_{L^\infty(B_R)} \leq CR^\Pi\quad \textrm{for all} \quad\frac{1}{2r_m} \ge R \geq 1
\end{equation}
and a bound for $L_{k_m}(r_m)$ and $Q_{k_m}(r_m)$ given by
\begin{equation}
\label{eq.qbound}
L_{k_m}(r_m)\le C\theta_m, \qquad Q_{k_m}(r_m) \leq C\theta_m.
\end{equation}

Observe that by definition of $\phi_m$ (and homogeneity of $u_0^{(k_m)}$) we have the orthogonality condition
\begin{equation}
\label{eq.ort_cond}
\int_{B_1} v_m(x) u_0^{(k_m)}(x)\, dx = 0,\qquad\text{and}\qquad \int_{B_1} v_m(x) x_n\, dx = 0,
\end{equation}
and simply by definition we have 
\begin{equation}
\label{eq.contradict}
\|v_m\|_{L^\infty(B_1)} = 1.
\end{equation}

Let us show \eqref{eq.boundu}-\eqref{eq.growthctrl}-\eqref{eq.qbound}.  We start by defining the monotone function 
\[
\theta(r) := \sup_{k\in \N}\sup_{\rho  >r} \left\{\rho^{-\Pi}\|u_k - \phi_{k, \rho}\|_{L^\infty(B_\rho)}\right\},
\]
so that, for $r > 0$, $\theta(r) < \infty$, and $\theta(r) \uparrow \infty$ as $r\downarrow 0$. Take now $r_m$ and $k_m$ subsequences such that 
\[
r_m^{-\Pi} \|u_{k_m} - \phi_{k_m, r_m}\|_{L^\infty(B_{r_m})} =:\theta_m \ge \frac{\theta(r_m)}{2}\to \infty
\]
as $r_m\downarrow 0$. 

From the definition of $\phi_{k, r}$, $\phi_{k, 2r} - \phi_{k, r} = \big(Q_k (2r) - Q_k (r) \big) u_0^{(k)} + (L_k(2r) - L_k(r)) x_n$, where we recall that $u_0^{(k)}$ is $1+\alpha_k$ homogeneous, and thus arguing as in \eqref{eq.usingthat} we have
\begin{align*}
\max\{|Q_k(2r)-&Q_k(r)|r^{1+ \alpha_k}, |L_k(2r)-L_k(r)|r\} \le C\|\phi_{2r}- \phi_r\|_{L^\infty(B_r)}
\\
& \leq  C\|\phi_{k, 2r} - u_k\|_{L^\infty(B_{2r})}+ C\|\phi_{k, r} - u_k\|_{L^\infty(B_r)}\leq Cr^\Pi \theta(r).\\
\end{align*}

Proceeding inductively, if $R = 2^N$, then
\begin{equation}
\label{eq.longeq}
\begin{split}
\frac{r^{1+\alpha_k -\Pi} |Q_k(Rr)- Q_k(r)|}{\theta(r)}& \leq C\sum_{j = 0}^{N-1} 2^{j(\Pi - 1- \alpha_k)} \frac{{(2^jr)}^{1+\alpha_k-\Pi} |Q_k(2^{j+1}r)- Q_k(2^jr)|}{\theta(r)} \\
& \leq C \sum_{j = 0}^{N-1} 2^{j(\Pi - 1- \alpha_k)} \frac{\theta(2^j r)}{\theta(r)} \leq C 2^{N(\Pi - 1-\alpha_k)} = CR^{\Pi - 1- \alpha_k}.
\end{split}
\end{equation}

Similarly with the terms $L_k(Rr)$ we have 
\[
\frac{r^{1-\Pi}}{\theta(r)} |L_k(Rr) - L_k(r)|\le C R^{\Pi - 1}.
\]
Thus, we obtain a bound on the growth control of $v_m$ given by \eqref{eq.growthctrl}. Indeed,
\begin{align*}
\|v_m\|_{L^\infty(B_R)} & = \frac{\|u_{k_m} - L_{k_m}(r_m) x_n - Q_{k_m}(r_m) u_0^{(k_m)}\|_{L^\infty(Rr_m)}}{\|u_{k_m} - \phi_{k_m, r_m}\|_{L^\infty(B_{r_m})}}\\
& \leq \frac{2}{\theta(r_m) r_m^\Pi} \|u_k - L_{k_m}(Rr_m) x_n - Q_{k_m}(Rr_m) u_0^{(k_m)}\|_{L^\infty(Rr_m)} +\\
&~~~~~~~~~~~~~~~~+\frac{2}{\theta(r_m) r_m^\Pi} |L_{k_m}(Rr_m) - L_{k_m}(r_m) | Rr_m\\
&~~~~~~~~~~~~~~~~+\frac{2}{\theta(r_m) r_m^\Pi} |Q_{k_m}(Rr_m) - Q_{k_m}(r_m) | (Rr_m)^{1+\alpha_k}\\
& \leq 2\frac{R^\Pi \theta(Rr_m)}{\theta(r_m)} + CR^\Pi,
\end{align*}
and now \eqref{eq.growthctrl} follows from the monotonicity of $\theta$.

Notice also that the previous computation in \eqref{eq.longeq} also gives a bound for $L_k(r)$ and $Q_k(r)$ given by
\[
L_k(r)\le C\theta(r), \qquad Q_k(r) \leq C\theta(r),
\]
which follows by putting $R = r^{-1}$. This gives \eqref{eq.qbound}.
\\[0.1cm]
{\it Step 4: Convergence of the blow-up sequence}. We have that (since $Q_{k_m}(r_m) > 0$, by \eqref{eq.Qpositive}), for any $\beta \ge 0$, 
\begin{equation}
\label{eq.equation}
\mathcal{P}_{\lambda, \Lambda}^- (D^2 (v_m - \beta u_0^{(k_m)}))  \le 0 \le \mathcal{P}_{\lambda, \Lambda}^+ (D^2 (v_m - \beta u_0^{(k_m)}))  
\end{equation}
in $B_1\setminus (\Omega_{k_m}'\cup \{x\cdot \nu_{k_m} \le 0, x_n = 0\})$ and also
\[
v_m - \beta u_0^{(k_m)} = 0\qquad\text{on}\qquad \Omega_{k_m}'\cap\{x\cdot \nu_{k_m}  \le 0, x_n = 0\}.
\]

Let us define $U_m^+ := \Omega_{k_m}'\cup \{x\cdot \nu_{k_m}\le 0, x_n = 0\}$. Observe that, from $C^{1,\tilde \alpha}$ regularity of solutions to fully nonlinear equations and the fact that $\Omega'_{k_m}$ is $C^{1,\alpha}$, we have that $|\nabla_{x'}u_{k_m}|\le Cr^{\tilde\alpha(1+\alpha)}$ in $B_r\cap U_m^+$. Similarly, using  \eqref{eq.qbound} we have that  $|\nabla_{x'}\phi_{m}|\le C\theta_m r^{\tilde\alpha(1+\alpha)}$ in $B_r\cap U_m^+$. Combining this we get that
\[
\|\nabla_{x'}u_{k_m} - \nabla_{x'}\phi_{m}\|_{L^\infty(B_r\cap U_m^+)}\le C \theta_m r^{\tilde\alpha(1+\alpha)}\qquad\text{for all}\quad r < \frac14,
\]
and 
\[
[\nabla_{x'} u_{k_m} - \nabla_{x'} \phi_{m}]_{C^{\tilde\alpha}(B_r\cap U_m^+)}\le C\theta_m \qquad\text{for all}\quad r < \frac14.
\]
 Observe that by the almost optimal regularity of solutions (Proposition~\ref{prop.opt}) and by taking $\eta_\circ$ small (so that the domain is almost flat) we can assume that $\tilde\alpha$ is arbitrarily close to $\alpha_k$. In particular, if $k$ is large enough, we can choose also $\sigma$ small enough so that $\Pi - 1 < \tilde\alpha(1+\alpha)$ in $B_{1/4}$.

In all, interpolating the previous two inequalities and using that $\Pi - 1 < \tilde\alpha(1+\alpha)$ in $B_{1/4}$ we obtain that 
\[
[\nabla_{x'} u_{k_m} - \nabla_{x'} \phi_{m}]_{C^{\gamma}(B_r\cap U_m^+)}\le C\theta_m r^{\Pi - 1}\qquad\text{for all}\quad r < \frac14,
\]
for some $\gamma > 0 $.

In particular, taking $r=r_m$, we get 
\[
[\nabla v_m]_{C^{\gamma}(B_1\cap r_m^{-1}U_m^+)}\le C.
\]
Repeating the same argument for any $R \ge 1$ with $r_m R < \frac14$, 
\[
[\nabla v_m]_{C^{\gamma}(B_R\cap r_m^{-1}U_m^+)}\le C(R).
\]

Together with \eqref{eq.equation} for $\beta =0$ we have that $v_m$ (which is continuous) is $C^{1,\gamma}$ in $B_R\cap r_m^{-1}U_m^+$ and satisfies a non-divergence form equation outside. A barrier argument combined with interior estimates gives
\[
[v_m]_{C^{\gamma}(B_R)}\le C(R)
\]
for some possibly smaller $\gamma > 0$, for $C(R)$ independent of $m$. 

In all, we can take $m \to \infty$ and $v_m$ converge, up to subsequences, locally uniformly to some $v_\infty\in C^\gamma_{\rm loc}(\R^n)$. Moreover, up to taking a further subsequence, we also assume $u_0^{(k_m)}$ converges locally uniformly to $u_0^\infty$, and $\nu_k$ to $\nu_\infty$, so that $\partial_e u_0^\infty = 0$ for all $e\in \mathbb{S}^{n-2}$ with $e\cdot \nu_\infty = 0$. Taking \eqref{eq.equation} to the limit, we get that, for all $\beta \ge 0$,
\begin{equation}
\label{eq.nondiv_eq}
\mathcal{P}_{\lambda, \Lambda}^- (D^2 (v_\infty - \beta u_0^{\infty}))  \le 0 \le \mathcal{P}_{\lambda, \Lambda}^+ (D^2 (v_\infty - \beta u_0^{\infty})) 
\end{equation}
in $B_1\setminus \{x\cdot \nu_\infty \le 0, x_n = 0\}$, and $v_\infty = 0$ in $\{x\cdot \nu_\infty \le 0, x_n = 0\}$. Finally, from \eqref{eq.growthctrl}, 
\[
\|v_\infty\|_{L^\infty(B_R)}\le C R^{\Pi}\quad\text{for all}\quad R\ge 1. 
\]

{\it Step 5: Two dimensional solution.} 
From the regularity in Proposition~\ref{prop.opt}, we know that for any $\eps > 0$, there exists $m$ large enough and $C$ depending only on $F$, $\alpha$, and $\eps$, such that
\[
|\partial_{\nu_{k_m}} u_{k_m}| \le C|x|^{\alpha_{k_m} - \eps}\quad \text{in}\quad B_{r_m}.
\]
We are using here that, for $m$ large enough, the normal to $\Omega_{k_m}'$ in $B_{r_m}$ does not vary too much, thus the corresponding regularity does not vary too much either. 

On the other hand, by the boundary Harnack inequality in slit domains for equations in non-divergence form (\cite{DS20} or \cite[Theorem 1.2]{RT21}) we know that there exists some $\beta > 0$ depending only on $\alpha$, and $F$ such that 
\[
w_m:= \frac{\partial_e u_{k_m}}{\partial_{\nu_{k_m}} u_{k_m}}\in C^\beta(B_{1/2})\quad\text{for all}\quad e\in \mathbb{S}^{n-2} : e\cdot \nu_{k_m} = 0,
\]
with estimates.
In particular, since the normal at the free boundary at 0 is $\nu_{k_m}$, $w_m(0) = 0$ and together with the almost optimal regularity in Proposition~\ref{prop.opt} we obtain 
\begin{equation}
\label{eq.fromharnack}
|\partial_{e} u_{k_m}|\le C |x|^{\alpha_{k_m} -\eps+\beta} \le C |x|^{\alpha_{k_m} +\beta/2} \quad\text{in}\quad B_{r_m},\quad\text{for all}\quad e\in \mathbb{S}^{n-2} : e\cdot \nu_{k_m} = 0,
\end{equation}
by choosing $\eps > 0$ small enough, and for some $C$ that now depends only on $\alpha$, and $F$. Thus 
\begin{equation}
\label{eq.unif_vm}
{\|\partial_e v_m\|_{L^\infty(B_1)}}\le \frac{r_m^{1-\Pi +\alpha_{k_m}+\beta/2}}{\theta_m} \le \frac{r_m^{\beta/8}}{\theta_m} \le r_m^{\beta/8}
\end{equation}
if we choose $\sigma \le \beta/ 8$ and we let $m$ large enough (recall $\alpha_{k_m}\to \alpha_\infty$). In particular, taking $m\to \infty$, the limiting function $v_\infty$ is two dimensional, and 
\[
v_\infty(x) = \tilde v_\infty(\nu_\infty\cdot x', x_n)
\]
 for some $\tilde v_\infty:\R^2\to \R$. 

{\it Step 6: Higher order estimates.} Let us fix some $m\in \N$, and let us assume without loss of generality that $\nu_{k_m} = \be_1$. 

Notice that $v_m$ is locally $C^{2,\alpha}$ outside of $U_m^+ = \Omega_{k_m}'\cup\{x_n = 0, x_1 \le 0\}$, in particular, it satisfies an equation with bounded measurable coefficients  there:
\[
\sum_{i, j = 1}^n a^m_{ij}(x) \partial_{ij} v_m(x) = 0 \quad\text{in}\quad B_1\setminus U_m^+,
\]
where $A_m(x) = (a^m_{ij}(x))_{i,j = 1}^n \in \mathcal{M}_n^S$ is uniformly elliptic with ellipticity constants $\lambda$ and $\Lambda$.  Proceeding as in the proof of Lemma~\ref{lem.divergenceform} (see \cite[Chapter 4]{FR21}) we can differentiate with respect to $\be_1$ and rewrite it as 
\begin{equation}
\label{eq.maineq}
{\rm div}_{1,n}\left(\tilde A_m(x) \nabla_{1,n} \partial_1 v_m \right) = -\partial_1\left\{\sum_{j = 2}^{n-1} \left(\sum_{i = 1}^n \frac{a^m_{ij}(x)}{a^m_{nn}(x)} \partial_ i \right) \partial_{j} v_m(x)\right\}  \quad\text{in}\quad B_1\setminus U_m^+,
\end{equation}
where 
\[
\tilde A_m(x) = \begin{pmatrix}
a^m_{11}(x) / a_{nn}^m(x) &  2a_{1n}^m(x) / a^m_{nn}(x)\\
0 & 1
\end{pmatrix} \in \mathcal{M}_2
\]
is uniformly elliptic, and where we can divide by $a^m_{nn}(x)$ since $A_m(x)$ is uniformly elliptic.  We have denoted here ${\rm div}_{1,n} w= \partial_1w+\partial_nw$ and $\nabla_{1,n} w = (\partial_1 w, \partial_n w)$. Notice, also, that all coefficients $a^m_{ij}(x) / a^m_{nn}(x)$ are uniformly bounded. 

If we fix $x_2 = \dots  = x_{n-1} =0$ in \eqref{eq.maineq}, and we denote 
\[
\tilde v_m(x_1, x_n) := v_m(x_1,0,\dots,0, x_n),
\]
then
\begin{equation}
\label{eq.tildef}
{\rm div}(\tilde A_m \nabla \partial_1 \tilde v_m) = \partial_1 \tilde f_m (x)\quad \text{in}\quad B_1\setminus \{x_n = 0, x_1 \le 0\} \subset \R^2,
\end{equation}
for some $\tilde f_m$. We will show $\tilde f_m\in L^p$ for some $p > 2$, independently of $m$. 

Let $x_\circ\in \{x_2 = \dots = x_{n-1} = 0\}\setminus \{x_n = 0, x_1\le 0\}$, and let $\rho := |x_\circ|$; we want to bound $|\nabla \partial_j u_{k_m}|(x_\circ)$ in terms of $\rho$, for $j \in \{2,\dots, n-1\}$. Notice that, on the one hand, we know
\begin{equation}
\label{eq.first}
\rho\|\partial_j u_{k_m}\|_{L^\infty(B_{\rho/8}(x_\circ))} \le C \rho^{1+\alpha_{k_m} +\beta/2}\quad \text{in}\quad B_{r_m},
\end{equation} 
from  \eqref{eq.fromharnack}.

On the other hand, if $\eta$ is universally small enough, we can either assume that $B_{\rho/4}(x_\circ) \subset B_1\setminus \Omega'_{k_m}$ or that $x_\circ \in B_{\rho/4}(x_\circ')$ with $x_\circ' = ((x_\circ)_1,0,\dots, 0)$ and $B_{\rho/2}(x_\circ')\cap \partial \Omega_{k_m}' = \varnothing$. 

If the first case holds, then by interior estimates for convex fully nonlinear equations on $u_{k_m}$, \cite{CC95}, we obtain that
\begin{equation}
\label{eq.second}
\rho^{2+\tilde \gamma} [\nabla \partial_j u_{k_m}]_{C^{\tilde\gamma}(B_{\rho/8}(x_\circ))} \le C \|u_{k_m}\|_{L^\infty(B_\rho(x_\circ))} \le C\rho^{1+\alpha_{k_m} -\eps}
\end{equation}
for some $C$ and any $\tilde\gamma > 0$ small enough depending only on $F$, and $\eps$, and where we are also using the almost optimal regularity for $u_{k_m}$, Proposition~\ref{prop.opt}. 

We can now interpolate \eqref{eq.first} and \eqref{eq.second} (see, for example, \cite[Section 6.8]{GT83}) and fix $\eps = \beta \tilde \gamma /4$ to obtain
\begin{equation}
\label{eq.boundvm}
|\nabla \partial_j u_{k_m}|(x_\circ) \le \|\nabla \partial_j u_{k_m}\|_{L^\infty(B_{\rho/8}(x_\circ))} \le C |x_\circ|^{\alpha_{k_m}-1 + \frac{\beta\tilde \gamma}{4(1+\tilde\gamma)}}
\end{equation}
for any $j\in \{2,\dots,n-1\}$. On the other hand, if the second case holds  (i.e. $x_\circ \in B_{\rho/4}(x_\circ')$ with $x_\circ' = ((x_\circ)_1,0,\dots, 0)$ and $B_{\rho/2}(x_\circ')\cap \partial \Omega_{k_m}' = \varnothing$), the same result \eqref{eq.boundvm} is obtained by means of boundary regularity estimates instead of interior estimates for convex fully nonlinear equations. 

Applying the previous bound \eqref{eq.boundvm} to $v_m$ we obtain that, given a point $\bar x_\circ\in \{x_2= \dots= x_{n-1} = 0\}$ and $j\in \{2,\dots, n-1\}$ (recall $\partial_j u_0^{(k_m)} \equiv 0$)
\[
|\nabla \partial_j v_m|(\bar x_\circ) \le r_m^{\tilde \beta/4} |\bar x_\circ|^{\alpha_{k_m} - 1 +\tilde \beta}
\]
if $\sigma$ is small enough  depending on $\beta$ and $\tilde \gamma$, and where $\tilde \beta$ is small (again depending on $\beta$ and $\tilde\gamma$).  In particular, we have 
\[
|\nabla \partial_j v_m|(x_1,0,\dots,0, x_n) \in L^p(B_1),\quad B_1\subset \R^2,\quad j\in \{2,\dots,n-1\},
\]
for some $p > 2$. Hence, after recalling \eqref{eq.maineq}, we get that, in \eqref{eq.tildef},  $\tilde f_m\in L^p_{\rm loc}(\R^2)$ and 
\[
\|\tilde f_m\|_{L^p(B_1)}\le Cr_m^{\tilde \beta /4}
\] for some $C$ independent of $m$, and some $p > 2$.

On the other hand, since $v_m$ is the difference of two solutions to a thin obstacle problem, it has interior $C^{2,\alpha}$ regularity (that may depend on $m$). Proceeding as before (to get \eqref{eq.boundvm}) we have that 
\[
|\nabla \partial_1 u_{k_m}|(x_\circ) \le \|\nabla \partial_j u_{k_m}\|_{L^\infty(B_{\rho/8}(x_\circ))} \le C |x_\circ|^{\alpha_{k_m}-\eps -1}
\]
and by homogeneity the same holds for $\phi_m$, for a constant $C$ that depends on $m$. In particular, 
\[
\|\nabla \partial_1 \tilde v_m\|_{L^p(B_1)}\le C_m,
\]
for some $C_m \uparrow \infty$ as $m\to \infty$. Thus, $\nabla \partial_1 \tilde v_m\in L^p_{\rm loc}(\R^2)$ for some $p> 2$, and we can apply Lemma~\ref{lem.weird_reg_est}. 

That is, by the regularity estimates from Lemma~\ref{lem.weird_reg_est} we get that that $\partial_1\tilde v_m \in C^{\delta}(B_1)$ for some $\delta > 0$ with estimates of the form 
\[
[\partial_1\tilde v_m]_{C^\delta(B_{1/2})} \le C\left( \|\partial_1\tilde v_m\|_{L^\infty(B_1)} + 1\right) 
\]
for some $C$ depending only on $F$, and $\alpha$. Using interpolation (see \cite{GT83}) and a standard covering argument to reabsorb the right-hand side (see for example \cite[Lemma 2.26]{FR21}) we now get that 
\[
[\partial_1\tilde v_m]_{C^\delta(B_{1/2})}\le C  (\|\tilde v_m\|_{L^\infty(B_1)} +1) \le C.
\]
Recall that $\tilde v_m(0) = |\nabla \tilde v_m(0)| = 0$. In particular, $v_m$ is uniformly $C^{1,\delta}$ for some $\delta > 0$ on $\{x_2 =\dots = x_{n-1} = x_n = 0\}$. That is, the limiting function $v_\infty$ is $C^{1,\delta}$ on $\{x_2 = \dots = x_{n-1} = x_n = 0\}$. 

Since $v_\infty$ satisfies an equation, \eqref{eq.nondiv_eq} with $\beta = 0$, and it is two-dimensional, boundary regularity estimates imply that $v_\infty$ is $C^{1,\delta}$ in each part $\{\pm x_1 > 0\}$, by taking $\delta$ smaller if necessary.
\\[0.1cm]
{\it Step 7: Conclusion.} Notice that, in the previous step, \eqref{eq.tildef} also holds exchanging $\tilde v_m$ by $\tilde v_m - \beta \tilde u_0^{(k_m)}$ for any fixed $\beta > 0$, and for any ball $B_R$ with $R \le \frac{1}{2r_m}$,
\[
{\rm div}\left(\tilde A_{m, \beta} \nabla \left(\partial_1 \tilde v_m - \beta \partial_1 \tilde u_0^{(k_m)}\right)\right) = \partial_1 \tilde f_m(x)\quad \text{in}\quad B_R\setminus \{x_n = 0, x_1 \le 0\},
\]
where $\|\tilde f_m\|_{L^p(B_{R})}\le C(R) r_m^{\tilde \beta/4}$, and $\tilde A_{m, \beta}$ is uniformly elliptic with ellipticity constants depending on $\lambda$ and $\Lambda$. 

In particular, we can take the limit $m\to\infty$ and observe that $f_m\downarrow 0$ and that, by homogenization theory, there exists $A^*_\beta(x)$ uniformly elliptic with ellipticity constants depending only on $\lambda$ and $\Lambda$ such that $\tilde A_{m,\beta}$ $H$-converge to $A^*_\beta$ up to subsequences (see \cite[Theorem 1.2.16, Proposition 1.2.19]{All02}). That is, the limit $\partial_1 \tilde v_\infty - \beta \partial_1 \tilde u_0^{(\infty)}$ satisfies 
\[
{\rm div} \left( A^*_\beta\left( \partial_1 \tilde v_\infty - \beta \partial_1 \tilde u_0^{(\infty)}\right) \right) = 0 \quad \text{in}\quad B_R\setminus \{x_n = 0, x_1 \le 0\}.
\]
Since this is true for any $\beta > 0$, we can proceed exactly as in the proof of Proposition~\ref{prop.uniqueness} to deduce that $\partial_1 \tilde v_\infty = \bar \beta \partial_1 \tilde u_0^{(\infty)}$ everywhere for some $\bar\beta \ge 0$ (by means of the boundary Harnack, Theorem~\ref{thm.boundaryHarnack}). That is, $\tilde v_\infty -\bar\beta \tilde u_0^{(\infty)} = g(x_n)$. Recalling that, outside of the contact set, $\tilde v_\infty -\bar\beta \tilde u_0^{(\infty)} $ satisfies an equation in non-divergence form with measurable coefficients, we deduce that $g(x_n) = \kappa x_n$ for some $\kappa\in \R$. 

In all, we have 
\[
\tilde v_\infty(x_1, x_n) = \bar\beta \tilde u_0^{(\infty)}(x_1, x_n) + \kappa x_n.
\]
Passing to the limit \eqref{eq.ort_cond} we deduce that $\bar\beta = \kappa = 0$ (we also use \eqref{eq.diff}, so that this is the only solution), and hence $\tilde v_\infty(x_1, x_n) \equiv 0$, contradicting \eqref{eq.contradict} in the limit. 
\end{proof}

As a consequence of the previous proposition, we directly have that there is always an expansion around regular points, where the first term is the (unique) blow-up at the point.

\begin{proof}[Proof of Theorem~\ref{thm.expansion}]
We just directly apply Proposition~\ref{prop.opt_reg} at a sufficiently small scale, where \eqref{eq.closeness} holds by uniqueness of blow-ups. 
\end{proof}

\section{Optimal regularity}
\label{sec.5}

The ideas of this section are based on those developed in \cite{FRS23} (see also \cite[Section~4.5]{Book-nonlocal}). 

Let us start by proving the following lemma, that will be crucial in the proof of the optimal regularity.

\begin{lem}
\label{lem.first}
Let $F$ satisfy \eqref{eq.F}. For every $\eps > 0$ there exists $\delta > 0$ depending only on $\eps$ and $F$ such that the following statement holds.

Let $u$ solve a fully nonlinear thin obstacle problem in $B_{1/\delta}$ with operator $F$, and let us assume that 0 is a free boundary point with $\nabla u(0) = 0$. Namely, 
\begin{equation}
  \label{eq.thinobst_delta}
  \left\{ \begin{array}{rcll}
  F(D^2u)&=&0& \textrm{ in } B_{1/\delta}\setminus \{x_n = 0, u=0\}\\
  F(D^2u)&\leq&0& \textrm{ in } B_{1/\delta}\\
  u&\geq&0& \textrm{ on } B_{1/\delta}\cap\{x_n=0\}. \\
  \end{array}\right.
\end{equation}

Let us also suppose 
\[
|F(A) - F^*(A)|\le \delta (1+\|A\|),\quad\text{for all}\quad A \in \mathcal{M}^S_n,
\]
where $F^*$ is the recession function of $F$, 
and
\begin{equation}
\label{eq.hyp2}
D^2_{x'} u \ge -\delta {\rm Id}\quad\text{in $B_{1/\delta}$} \quad \text{and}\quad \|u\|_{L^\infty(B_R)}\le R^{2-\eps_\circ}\quad\text{for all}\quad 1\le R\le 1/\delta, 
\end{equation}
where $\eps_\circ$ is small depending only on $F$ (as defined in \eqref{eq.reg_points}).

Then we have
\begin{equation}
\label{eq.concl_lem}
\|u - \|u \|_{L^\infty(B_1)} u_0\|_{C^1(B_1^+\cup B_1^-)} \le \eps,
\end{equation}
where $u_0$ is a regular blow-up for the operator $F$ with $\|u_0\|_{L^\infty} = 1$ (as constructed in Lemma~\ref{lem.existence}, with $G$ of the form \eqref{eq.givenby} for $F$). 
\end{lem}
\begin{proof}
Arguing by contradiction, suppose that $u_k$ satisfy the hypotheses with $\delta = 1/k$ but the thesis fails for some $\eps > 0$. Then $u_k$ converges to some $u_\infty$ convex in the directions parallel to the thin space and satis\-fying a fully nonlinear thin obstacle problem globally with some operator $F_\infty$ 1-homogeneous.  We then get a contradiction by the classification of global subquadratic and convex solutions \cite[Theorem 4.1]{RS17} and the uniqueness of blow-ups.
\end{proof}

The following is a continuation of the previous lemma, where we further assume a non-degenerate blow-up. 

\begin{lem}
\label{lem.second}
Let $F$ satisfy \eqref{eq.F} and be 1-homogeneous. There exists a universal $\delta_\circ$ depending only on $F$ such that if $u$ satisfies \eqref{eq.thinobst_delta}-\eqref{eq.hyp2} with $\delta \le \delta_\circ$, and $\|u \|_{L^\infty(B_1)}\ge \frac12$, then 
\begin{equation}
\label{eq.boundexp}
|u(x) |\le C |x|^{1+{\alpha_F}}\quad\text{for all $x\in B_1$},
\end{equation}
for some $C$ depending only on $F$, and where ${\alpha_F}$ is given by \eqref{eq.alphan_opt}. 
\end{lem}
\begin{proof}
Let $\eps > 0$ to be chosen, and let $\delta_\circ := \delta(\eps)$ given by Lemma~\ref{lem.first}, which will become universal once $\eps > 0$ is fixed. 

Then, since $\|u \|_{L^\infty(B_1)}\ge \frac12$, we will show that 0 is a regular free boundary point, the free boundary is $C^{1,\alpha}$ in $B_{\frac12}$ with $C^{1,\alpha}$ norm bounded by $C \eps$, and 
\begin{equation}
\label{eq.onlytocheck}
\left\|\frac{u(rx)}{\|u\|_{L^\infty(B_r)}} -  u_0(x)\right\|_{L^\infty(B_1)}\le \eps \quad\text{for all}\quad r < 1,
\end{equation}
where $u_0$ is a unit norm blow-up at 0 given by Corollary~\ref{cor.mainprop}. We now fix $\eps$ small enough (depending only on $F$), such that the hypotheses of Proposition~\ref{prop.opt_reg} hold. In particular, we have an expansion of the form 
\[
|u(x) - Q u_0^{(\nu)}(x)| \le C |x|^{1+\alpha_\nu + \sigma/2}\quad\text{in}\quad B_1
\]
for some $Q$ bounded depending only on $F$. Since ${\alpha_F} \le \alpha_\nu$ by definition and $u_0^{(\nu)}$ is $1+\alpha_\nu$ homogeneous and has norm 1, the bound \eqref{eq.boundexp} directly holds. 

Let us then show that indeed 0 is a regular free boundary point, the free boundary is $C^{1,\alpha}$ in $B_{\frac12}$ with $C^{1,\alpha}$ norm bounded by $C \eps$, and \eqref{eq.onlytocheck} holds.

From 
\eqref{eq.concl_lem}, the fact that 0 is a regular point  and the free boundary regularity around it are standard once we have a boundary Harnack for slit domains (cf. \cite[Section 5]{Fer21}). Indeed, suppose that $u_0 = u_0(x_1, x_n)$ (hence the normal vector to the free boundary is $\nu = \be_1$). Then, if $\eps$ is small enough, a standard argument 
  gives that $\partial_1 u \ge 0$ in $B_{1/2}$ (for example, using \cite[Lemma 5]{ACS08}). The same holds for directions in a cone around $\be_1$, so that in fact the free boundary is Lipschitz, with Lipschitz constant arbitrarily small (depending on $\eps> 0$), cf. \cite[Proposition 5.1]{RT21}. 

On the other hand, let $e\in\mathbb{S}^{n-1}\cap\{x_n = 0\}$ be such that $e\cdot \be_1 = 0$. Then $\partial_e u_0 = 0$ and therefore, by assumption, $\|\partial_e u\|_{L^\infty(B_1)}\le \eps$. By applying now the boundary Harnack for slit domains for equations in non-divergence form, \cite[Theorem 4.2]{RT21}, we get that 
\[
\left\|\frac{\partial_e u}{\partial_1 u}\right\|_{C^\alpha(B_{1/2})} \le C\eps.
\]
This gives that the free boundary is $C^{1,\alpha}$ with norm bounded by $\eps$ (as in \cite[Theorem 5.3]{Fer21}). 
%
%
Observe, also, that we can choose $u_0$ to be arbitrarily close (depending on $\eps$) to the blow-up at 0, $u_0^{(\nu)}$, where $\nu$ is the unit outward normal at the contact set at the origin, and $u_0^{(\nu)}$ is defined as in the proof of Corollary~\ref{cor.mainprop}.  

Thus, there only remains to show \eqref{eq.onlytocheck}. Let us define $u_\rho(x) := u(\rho x) /\|u\|_{L^\infty(B_\rho)}$ for $\rho < 1$. Observe that $u_\rho$ also satisfies a fully nonlinear thin obstacle problem with the same $F$ (since it is 1-homogeneous), has norm 1, and in a uniform ball around the origin the free boundary is formed exclusively of regular points. Proceeding by a barrier argument from below (as in \cite[Proposition 7.1]{RS17} to get a non-degeneracy condition) we deduce that 
\begin{equation}
\label{eq.touse}
\frac{\|u\|_{L^\infty(B_{\rho \bar r})}}{\|u\|_{L^\infty(B_{\rho})}} = \|u_\rho\|_{L^\infty(B_{\bar r})}\ge C {\bar r}^{2-\beta_\circ}\quad\text{for all $\bar r < 1$},
\end{equation}
for some $\beta_\circ > 0$ depending only on $F$. 

Let us now consider $u_{r}$ for $r < 1$. Notice that $u_r$ satisfies \eqref{eq.thinobst_delta}. Moreover, 
\[
D^2_{x'} u_r = \frac{r^2}{\|u\|_{L^\infty(B_r)}} (D^2 u) (rx) \ge -\delta{\rm Id}\quad\text{in}\quad B_{1/\delta}
\]
where we are using \eqref{eq.touse} with $\rho = 1$. On the other hand, if $Rr \ge 1$,
\[
\|u_r\|_{L^\infty(B_R)} = \frac{\|u\|_{L^\infty(B_{rR})}}{\|u\|_{L^\infty(B_r)} } \le \frac{CR^{2-\eps_\circ} r^{2-\eps_\circ}}{\|u\|_{L^\infty(B_r)} } \le CR^{2-\eps_\circ} r^{\beta_\circ-\eps_\circ} \le R^{2-\eps_\circ},
\]
where we are taking $\eps_\circ$ smaller if necessary (depending only on $F$), and we are using again \eqref{eq.touse} with $\rho = 1$; and if $Rr < 1$ with $R>1$, 
\[
\|u_r\|_{L^\infty(B_R)} =  \frac{\|u\|_{L^\infty(B_{rR})}}{\|u\|_{L^\infty(B_r)} } \le CR^{2-\beta_\circ} \le R^{2-\eps_\circ}
\]
where we are using \eqref{eq.touse} with $\bar r = \frac{1}{R}$ and $\rho = R r$. In all, $u_r$ satisfies the hypotheses of Lemma~\ref{lem.first}, and so 
(since $u_r$ has norm 1) we have, for all $r < 1$,
\[
\|u_r -  u_0\|_{L^\infty(B_1)}\le \eps.
\]
By taking $\eps$ smaller if necessary, we can assume, as before, that $ u_0 = u_0^{(\nu)}$ is the same for all $r < 1$, where $u_0^{(\nu)}$ is the unique blow-up at zero, with homogeneity $1+\alpha_\nu$. Thus, \eqref{eq.onlytocheck} holds.
\end{proof}

We can now prove the optimal regularity in the 1-homogeneous case.

\begin{thm}
\label{thm.opt5}
Let $u$ be a solution to the fully nonlinear obstacle problem, \eqref{eq.thinobst}, with operator $F$ given by \eqref{eq.F} and 1-homogeneous, and let us assume that 0 is a free boundary point.  Let ${\alpha_F}\in (0, 1)$ be given by \eqref{eq.alphan_opt}. 

Then, $u\in C^{1,{\alpha_F}}(B_1^+)\cap   C^{1,{\alpha_F}}(B_1^-)$ and
\[
\|u\|_{C^{1+{\alpha_F}}(B_{1/2}^+)}+\|u\|_{C^{1+{\alpha_F}}(B_{1/2}^-)} \le C \|u\|_{L^\infty(B_1)}
\]
for some $C$ depending only on $F$. 
\end{thm}
\begin{proof}
Let us assume $\|u\|_{L^\infty(B_1)}= 1$. By the semiconvexity in the directions parallel to the thin space (see \cite[Proposition 2.2]{Fer16}) we have 
\[
D^2_{x'}u \ge -C{\rm Id}
\]
for some $C$ depending only on $F$. In particular, if $\delta_\circ$ is given by Lemma~\ref{lem.second}, there exists a rescaling $r_0\le \delta_\circ$ (depending only on $F$) such that the function $u(r_0 \,\cdot\,)$ satisfies \eqref{eq.thinobst_delta}-\eqref{eq.hyp2} with $\delta =\delta_\circ$.
We claim now that, for $C_0:= \max\{ C, r_0^{-1-\alpha_F}\}$, 
\begin{equation}
\label{eqn:optimalclaim}
|u(x) | \leq C_0 |x|^{1+\alpha_F} \qquad \mbox{ for all } \quad x\in B_{r_0}.
\end{equation}
To prove this, consider the set of $r<r_0$ such that the following inequality fails
\begin{equation}
\label{eqn:boh}
\| u \|_{L^\infty(B_r)} \leq \Big(\frac r {r_0} \Big)^{1+\alpha_F}.
\end{equation}
If this set is empty, \eqref{eqn:optimalclaim} holds. Otherwise, take $r_1$ its supremum and observe that
 $\| u \|_{L^\infty(B_{r_1})} = {r_1}^{1+\alpha_F} {r_0}^{-1-\alpha_F}.$ Hence we can apply Lemma~\ref{lem.second} to 
 $$u_{r_1}(x) := \frac{u(r_1x)}{\|u\|_{L^\infty(B_{r_1})}}=  \frac{u(r_1x){r_0}^{1+\alpha_F}}{{r_1}^{1+\alpha_F}}.
$$
We observe that $\| u_{r_1}\|_{L^\infty(B_1)} = 1$ and that the assumption \eqref{eq.hyp2}
 is satisfied, 
 since $u(r_0 \,\cdot\,)$ satisfies it and \eqref{eqn:boh} holds for $r_1<r<r_0$. By Lemma~\ref{lem.second} we deduce that $|u_{r_1}(x) | \leq C_0 |x|^{1+\alpha_F}$ for every $x\in B_1$ and hence \eqref{eqn:optimalclaim} holds.

Combining \eqref{eqn:optimalclaim}  with interior (and boundary) $C^{2,\alpha}$ estimates for the convex fully nonlinear operator $F$, we get the desired result.
\end{proof}

We can now give the proof of Theorem~\ref{thm.main_1homog}:
\begin{proof}[Proof of Theorem~\ref{thm.main_1homog}]
It is a consequence of Theorem~\ref{thm.expansion} and Theorem~\ref{thm.opt5}. The growth at regular points is a consequence of the homogeneity of blow-ups. 
\end{proof}
%

%
%
%
%
%
%

\section{Rotationally invariant operators}
\label{sec.6}
In this section we will prove Proposition~\ref{prop.pucci_res} first, and then use it to prove Theorem~\ref{thm.rot_inv}.

Let us consider the thin obstacle problem for a Pucci operator $\mathcal{P}^+_{\lambda, \Lambda}$, 
\begin{equation}
  \label{eq.thinobst_P}
  \left\{ \begin{array}{rcll}
  \mathcal{P}^+_{\lambda, \Lambda}(D^2u)&=&0& \textrm{ in } B_1\setminus \{x_n = 0, u=0\}\\
  \mathcal{P}^+_{\lambda, \Lambda}(D^2u)&\leq&0& \textrm{ in } B_1\\
  u&\geq&0& \textrm{ on } B_1\cap\{x_n=0\}. \\
  \end{array}\right.
\end{equation}

We want to classify blow-ups at regular points for this problem. In order to do that, we will use the results in \cite{Leo17}, where the author gives explicit solutions and homogeneities for positive (and negative) solutions in planar cones. 

More precisely, given a cone in $\R^2$ with aperture $\theta\in (0, \pi)$, $\mathcal{C}_\theta\subset\R^2$, the unique positive and negative homogeneous solutions, $u_\pm$, to the problem 
\[
  \left\{ \begin{array}{rcll}
  \mathcal{P}^+_{\lambda, \Lambda}(D^2u_\pm)&=&0& \textrm{ in } \mathcal{C}_{\theta}\\
  u_\pm &=&0& \textrm{ on } \partial\mathcal{C}_\theta\\
    \pm u_\pm &\ge&0& \textrm{ in } \mathcal{C}_\theta
  \end{array}\right.
\]
have homogeneity $1+\alpha_\pm \ge 1$ given by the implicit equations
\[
g(\alpha_-, \omega) = \frac{\theta}{2},\qquad h(\alpha_+, \omega) = \frac{\theta}{2},\qquad \omega = \frac{\Lambda}{\lambda}. 
\]
The functions $g = g(\alpha, \omega), h = h(\alpha,\omega) :[1,\infty)\times[1,\infty) \mapsto [0,\pi/2]$ are defined as 
\[
g(\alpha, \omega) = \arctan{\sqrt{w}} + \frac{1-\alpha}{\sqrt{(\alpha+1/\omega)(\alpha+\omega)}}\arctan{\sqrt{\frac{\omega(\alpha+1/\omega)}{\alpha+\omega}}}
\]
and
\[
h(\alpha, \omega) = \arctan\frac{1}{\sqrt{w}} + \frac{1-\alpha}{\sqrt{(\alpha+1/\omega)(\alpha+\omega)}}\arctan{\sqrt{\frac{\alpha+\omega}{\omega(\alpha+1/\omega)}}}.
\]
(See \cite[Section 2]{Leo17}.) Furthermore, for every fixed $\omega$ the functions $g(\cdot, \omega)$ and $g(\cdot, \omega)$ are decreasing in $\alpha$; and for every fixed $\alpha$, the functions $g(\alpha,\cdot)$, $-h(\alpha, \cdot)$  are increasing in $\omega$. 

With these definitions and the construction in Lemma~\ref{lem.existence} we can now give the proof of Proposition~\ref{prop.pucci_res}. 

\begin{proof}[Proof of Proposition~\ref{prop.pucci_res}]
By the construction in Lemma~\ref{lem.existence} (and using the definitions of $g$ and $h$ above, from \cite{Leo17}), our desired solution satisfies 
\begin{equation}
\label{eq.masterblowup}
2g(\alpha, \omega)+ h(\alpha, \omega) = \pi. 
\end{equation}
That is, 
\begin{align*}
\frac{\pi}{2} & = \arctan {\sqrt{w}} + \frac{1-\alpha}{\sqrt{(\alpha+1/\omega)(\alpha+\omega)}}\left[\frac{\pi}{2} + \arctan{\sqrt{\frac{\omega(\alpha+1/\omega)}{\alpha+\omega}}}\right] \\
& =  g(\alpha, \omega) + \frac{\pi(1-\alpha)}{2\sqrt{(\alpha+1/\omega)(\alpha+\omega)}}.
\end{align*}

In particular, for each $\omega \ge 1$ fixed, it is the sum of two strictly decreasing functions in $\alpha$, and so there is at most one solution (which we know exists for $\alpha\in (0, 1)$).

In order to simplify the previous expressions, let us define $x := \frac{1}{\sqrt{\omega}}\in [0,1]$.
Equation \eqref{eq.masterblowup} then translates into 
\[
F(x, \alpha) := \frac{(1-\alpha)x}{\sqrt{(\alpha+x^2)(\alpha x^2+1)}}\left[\frac{\pi}{2} + \arctan{\sqrt{\frac{\alpha+x^2}{\alpha x^2+1}}}\right]-\arctan{x} = 0.
\]
That is, for each $x\in [0, 1]$ we are looking for an $\alpha\in (0, 1)$ such that $F(x, \alpha) = 0$. As before, $F(x, \alpha)$ is strictly decreasing in $\alpha\in (0, 1)$ for each fixed $x\in [0, 1]$. Moreover, $F(\cdot, \alpha)$ is a smooth function with $F(0, \alpha) = 0$ and $F(1, \alpha) = \frac{\pi}{2}\frac{1-2\alpha}{1+\alpha}$.

Let us now define 
\[
G(x, \alpha ):= \frac{\sqrt{(\alpha+x^2)(\alpha x^2+1)}}{(1-\alpha)x} F(x, \alpha), 
\]
so that $G(0, \alpha ) = \frac{\pi}{2}+\arctan{\sqrt{\alpha}} - \frac{\sqrt{\alpha}}{1-\alpha}$ and $G(1, \alpha) = \frac{\pi}{2}\frac{1-2\alpha}{1-\alpha}$. Notice that $G(x, \alpha)$ is smooth for $(x, \alpha) \in [0, 1]\times[0,1)$. 

A direct computation yields
\begin{align*}
\frac{\partial}{\partial x} G(x, \alpha) = \frac{\alpha  (1+x^2) \left(-x + \arctan(x) (1-x^2)\right) }{(1-\alpha) x^2 \sqrt{(\alpha+x^2)(\alpha x^2+1)}  }.
\end{align*}
In particular, since $\arctan{x} < \frac{x}{1-x^2}$ for $x\in (0, 1)$, we have that 
\[
\frac{\partial}{\partial x} G(x, \alpha) < 0\quad\text{for}\quad (x, \alpha) \in (0, 1)\times(0, 1).
\]

On the other hand, we can also compute
\[
\frac{\partial}{\partial \alpha} G(x, \alpha) = \frac{(1-x^2)(1-\alpha)^2x - \arctan(x) (x^2+1)^2 (\alpha+1)^2}{2(\alpha+1)\sqrt{(\alpha+x^2)(\alpha x^2 +1)}(1-\alpha)^2 x}.
\]
That is, for $\alpha\in (0, 1)$ and $x\in (0,1)$,
\[
\frac{\partial}{\partial \alpha} G(x, \alpha) < 0 \quad\Longleftrightarrow \quad \frac{(1-\alpha)^2}{(1+\alpha)^2} < \frac{\arctan{(x)}}{x} \frac{(x^2+1)^2}{1-x^2} . 
\]
In particular, notice that
\[
\frac{(1-\alpha)^2}{(1+\alpha)^2} < 1 < (1+x^2)^2 \frac{\arctan{(x)}}{x-\frac{x^3}{3}} \frac{(1-\frac{x^2}{3})}{(1-x^2)}  = \frac{\arctan{(x)}}{x} \frac{(x^2+1)^2}{1-x^2}
\]
for all $(x, \alpha) \in (0, 1)\times(0, 1)$. Thus, $\frac{\partial}{\partial \alpha} G(x, \alpha) < 0 $ for $(x, \alpha) \in (0, 1)\times(0, 1)$ as well. 

Hence, $G(x, \alpha)$ is (strictly) decreasing in both $x$ and $\alpha$ (recall that $v$ is $\alpha$-homogeneous if and only $G(x, \alpha) = 0$, where $x = \sqrt{\lambda/\Lambda}$). 

We know that for each $x\in (0, 1]$ there is a unique $\alpha\in (0, 1)$ such that $G(x, a) = 0$. In particular, there exists some function $h = h(x):(0, 1)\mapsto (0, 1)$ such that $G(x, h(x)) = 0$. Using that $G$ is strictly decreasing in both variables (alternatively, by the implicit function theorem), we have that $h$ is strictly decreasing.

That is, if $\alpha = \alpha(\omega)$ denotes the homogeneity (1+$\alpha$) of $v$, then $\alpha$ is strictly increasing. In particular, since $\alpha(1) = \frac12$, we reach the desired result. 

Finally, when $x\downarrow 0$, $G(x, \alpha)$ vanishes when $\alpha$ approaches the unique solution to 
$G(0, \alpha ) = \frac{\pi}{2}+\arctan{\sqrt{\alpha}} - \frac{\sqrt{\alpha}}{1-\alpha} = 0$.
\end{proof}
%

Before proving Theorem~\ref{thm.rot_inv}, let us show the following two-dimensional lemma, stating that rotationally invariant, 1-homogeneous, and convex  operators are equivalent (for solutions to thin obstacle problems) to some Pucci $\mathcal{P}^+_{\lambda, \Lambda}$. 

\begin{lem}
\label{lem.puccilem}
Let $F:\mathcal{M}_2\to F$ be a fully nonlinear operator with ellipticity constants $0< \lambda\le \Lambda < \infty$ and such that it is rotationally invariant. Suppose, also, that $F$ is 1-homogeneous and convex. Then, there exists some $1\le \tilde\omega\le \frac{\Lambda}{\lambda}$ such that $\{F \le 0 \} = \{\mathcal{P}^+_{1, \tilde \omega}\le 0\}$.

In particular, if $u:\R^2 \to \R$ is a two-dimensional solution to 
\begin{equation}
\label{eq.statement}
  \left\{ \begin{array}{rcll}
  F(D^2u)&=&0& \textrm{ in } B_1\setminus \{x_2= 0, u=0\}\\
  F(D^2u)&\leq&0& \textrm{ in } B_1\\
  u&\geq&0& \textrm{ on } B_1\cap\{x_2=0\}. \\
  \end{array}\right.
 \end{equation}
  Then, there exists some $1\le \tilde \omega\le \frac{\Lambda}{\lambda}$ such that $u$ satisfies
 \[
   \left\{ \begin{array}{rcll}
  \mathcal{P}^+_{1, \tilde \omega}(D^2u)&=&0& \textrm{ in } B_1\setminus \{x_2 = 0, u=0\}\\
  \mathcal{P}^+_{1, \tilde \omega}(D^2u)&\leq&0& \textrm{ in } B_1,
  \end{array}\right.
 \] that is, $u$ solves a thin obstacle problem with the Pucci operator $\mathcal{P}^+_{1, \tilde \omega}$. 
\end{lem}
\begin{proof}
Since $F$ depends only on the eigenvalues (or alternatively, it is rotation invariant), we can write $F(D^2 u) = f(\lambda_1(u), \lambda_2(u))$, where $\lambda_1(u) \ge \lambda_2(u)$ denote the two eigenvalues of $D^2 u$. Since $F$ is elliptic, whenever $F(D^2 u) = 0$ we have that $\lambda_1(u)\lambda_2(u) \le 0$, so we can assume $f$ to be defined in $D := [0, \infty)\times(-\infty, 0]\subset \R^2$. 

The fact that $F$ is uniformly elliptic implies that $f$ is strictly increasing in both components $\lambda_1$ and $\lambda_2$; and since $F$ is 1-homogeneous, we have that $f$ is 1-homogeneous as well. In particular, the zero level set of $f$ in $D$ is a cone. From the fact that $f$ is strictly increasing in both coordinates, the zero level set of $f$  is either empty in $D$ or a line passing through the origin, namely, $\{f = 0\}\cap D = \{(\lambda_1, -\tilde \omega \lambda_1)\}$ for some $\tilde \omega \ge 0$. 

Now observe that, if it is non-empty, the zero (and sub-zero) level set for $f$ coincides with the zero (and sub-zero) level set of a Pucci operator $\mathcal{P}^+_{1, \tilde \omega}$ if $\tilde \omega\ge 1$ or $\mathcal{P}^-_{\tilde \omega, 1}$ if $\tilde \omega < 1$. Hence, if we show the right inequalities (and existence) for $\tilde \omega$ we will be done. 

We have
\begin{equation}
\label{eq.ineq.op}
c \,{\rm tr} (A) \le F(A) \le \mathcal{P}^+_{\lambda, \Lambda}(A)\quad\text{for all}\quad A\in \mathcal{M}_2^S
\end{equation}
for some $c > 0$. The second inequality holds by definition of Pucci operator, while the first inequality is a consequence of the fact that $F$ is rotation invariant and convex. Indeed, since $F$ is convex, there exists some linear operator $L(A) := {\rm tr}(\tilde L A)$ for some $\tilde L\in \mathcal{M}_2^S$ such that $L(A) \le F(A)$, and since it is rotation invariant
\[
{\rm tr}(\tilde L O^{-1}AO) = {\rm tr}(O \tilde L O^{-1}A) \le F(A) \quad\text{for all}\quad A\in \mathcal{M}_2^S,
\]
and for all rotations $O\in \mathcal{O}(2)$. In particular, it holds for $O' = \begin{pmatrix} 0 &1 \\ -1 &0\end{pmatrix}$, and thus we also have 
\[
\frac12 {\rm tr}(L){\rm tr} (A)=  \frac12 {\rm tr}((L+ O' L (O')^{T})A)\le F(A)
\]
for all $A\in \mathcal{M}_2^S$ and hence \eqref{eq.ineq.op} holds. 

Finally, from \eqref{eq.ineq.op} we immediately deduce that $\tilde\omega$ exists and $1 \le \tilde \omega \le \frac{\Lambda}{\lambda}$, as we wanted to see. 
\end{proof}

The following proposition will now directly give the proof of Theorem~\ref{thm.rot_inv}.

\begin{prop}
\label{prop.prop_rot}
Let $F$ be a rotationally invariant operator of the form \eqref{eq.F}. Then, either $\{F \le 0 \} = \{{\rm tr}\, (\cdot) \le 0\}$ or $\{F^* \le 0\}\subset \{\mathcal{P}^+_{1,\Lambda_F} \le 0\}$ for some $\Lambda_F > 1$.
\end{prop}
\begin{proof}
We first remark that, by convexity, $\{F\le 0\} = \{{\rm tr}\, (\cdot)\le 0\}$ if and only if $\{F^*\le 0\} = \{{\rm tr}\, (\cdot)\le 0\}$. Since $F^*$ is a convex Hessian equation, we can write 
\[
F^*(A) = f^*(\ell_1(A),\dots,\ell_n(A))\quad \text{for } A\in \mathcal{M}_n^S,
\]
where  $\ell_1(A) \le \dots\le \ell_n(A)$ denote the ordered eigenvalues of $A$, and $f^*:\R^n\to \R$ is invariant under permutations of the coordinates, convex and elliptic, in the sense that for a dimensional constant $C_n$, and for any $i \in \{1,\dots,n\}$, $\mu > 0$
\[
\frac{\lambda}{C_n}\mu \le f^*(\ell_1,\dots,\ell_i+\mu,\dots, \ell_n) -f^*(\ell_1,\dots,\ell_i,\dots,\ell_n)\le \Lambda C_n \mu.
\]
 (See \cite{CNS85, NV13} for more details.) 

%
If $F^*(A)$ is a function of the trace of $A$, then $\{F^* \le 0 \} = \{{\rm tr}\, (\cdot) \le 0\}$ by ellipticity and the fact that $F^*(0) = 0$, and the proposition is proved.
Let us assume otherwise that $f^*(\ell_1, ...,\ell_n)$ is not a convex function of $\ell_1+...+\ell_n$. Then, there exists some $p_0\in\R^n$ such that $\nabla f^* (p_0) = (L_1,...,L_n)$ exists and is not a multiple of the vector $(1,...,1)$.
By the rearrangement inequality and ellipticity we can assume, up to changing the point $p_0$ into one of its coordinate permutations,
\begin{equation}
\label{eq.gammaellipt}
 0< \frac{\lambda }{ C_n }\le L_1 \le \dots \le L_n\le \Lambda C_n.
\end{equation}

We now claim that, for $\Lambda_F = 1+\frac 1 n\Big(\sqrt{\frac{L_n}{L_1}}-1\Big) > 1$,
\[
\{ F^* \le 0 \}\subset \{\mathcal{P}^+_{1,\Lambda_F} \le 0\}.
\]
Indeed, let $A\in \mathcal{M}^S_n$ with eigenvalues $\ell_1\le ...\le\ell_n$ such that $F^*(A) \le 0$, $\ell_1 \le 0, \ell_n > 0$, and let
\[
j:= \sup \{j\in \{1,\dots,n\} : \ell_j \le 0\}.
\]
Since $f^*$ is $1$-homogeneous,  $ (L_1,...,L_n)$ belongs to its subdifferential at $0$. Hence
\begin{equation}
\label{eq.either1}
F^*(A) = f^* (\ell_1 , \dots, \ell_n) \ge \sum_{i= 1}^n L_i \ell_i \ge L_j \left( \sum_{i = 1}^{n} \ell_i + \frac{L_1-L_j}{L_j} \ell_1+ \frac{L_n-L_j }{L_j} \ell_n\right).
\end{equation}
Notice that the last two contributions are nonnegative. We bound $F^*(A) $ from below by $\mathcal{P}^+_{1,\Lambda_F} (A)$, up to a constant, by studying two cases separately. If $L_j \leq \sqrt{L_1L_n}$, 
neglecting the second to last contribution in \eqref{eq.either1}, we have
\begin{equation*}
\frac{F^*(A)}{L_j}  \ge \sum_{i = 1}^{n} \ell_i + \Big(\sqrt{\frac{L_n}{L_1}}-1\Big) \ell_n  \ge  \sum_{i = 1}^{j} \ell_i + \Big(1+\frac 1 n\Big(\sqrt{\frac{L_n}{L_1}}-1\Big) \Big)\sum_{i = j+1}^{n} \ell_i  = \mathcal{P}^+_{1,\Lambda_F} (A).
\end{equation*}
If $L_j \geq \sqrt{L_1L_n}$, 
neglecting the last contribution in \eqref{eq.either1}, we have
\begin{equation*}
\begin{split}
\frac{F^*(A)}{L_j}  &\ge \sum_{i = 1}^{n} \ell_i + \Big(\sqrt{\frac{L_1}{L_n}}-1\Big) \ell_1  \ge \Big(1+\frac 1 n\Big(\sqrt{\frac{L_1}{L_n}}-1\Big) \Big) \sum_{i = 1}^{j} \ell_i + \sum_{i = j+1}^{n} \ell_i  
\\&\ge \Big(1+\frac 1 n\Big(\sqrt{\frac{L_1}{L_n}}-1\Big) \Big)\mathcal{P}^+_{1,\Lambda_F} (A).
\end{split}\end{equation*}
In both cases we have proved that $ \mathcal{P}^+_{1,\Lambda_F} (A)\le 0$.
\end{proof}

And thanks to the previous proposition we can now give the proof of Theorem~\ref{thm.rot_inv}. 

\begin{proof}[Proof of Theorem~\ref{thm.rot_inv}]
If $\{F \le 0 \} = \{{\rm tr}\, (\cdot) \le 0\}$, we are done by the results for the Signorini problem, \cite{AC04}. Let us suppose then that $\{F \le 0 \} \neq \{{\rm tr}\, (\cdot) \le 0\}$. 

We compute ${\alpha_F}$ according to the definition \eqref{eq.alphan_opt},
 \[{\alpha_F} := \min_{e\in \mathbb{S}^{n-2}} \alpha_{F_{(e)}} \in (0, 1),\]
 where $F_{(e)}$ are the two-dimensional blow-downs in the direction $e$. Observe that, since $F$ is rotational invariant, blow-ups at regular points satisfy a 2-dimensional fully nonlinear thin obstacle problem with operator $\mathcal{P}^+_{1,\tilde \omega}$ by Lemma~\ref{lem.puccilem}. Moreover, arguing as in the proof of Lemma~\ref{lem.puccilem}, since we have $\{ F^* \le 0 \}\subset \{\mathcal{P}^+_{1,\Lambda_F} \le 0\}$ by Proposition~\ref{prop.prop_rot}, up to a dimensional constant we have that $\tilde \omega -1>  c_n(\Lambda_F-1)$. 
 
 Hence, by Proposition~\ref{prop.pucci_res}, $\alpha_{F_{(e)}} > \frac12$ for all $e\in \mathbb{S}^{n-1}\cap \{x_n = 0\}$, and ${\alpha_F} > \frac12$. From the regularity result in Proposition~\ref{prop.opt}, we are done. 
\end{proof}

\section{Non-homogeneous operators}
\label{sec.7}
In this section we deal with the more general case of non-homogeneous operators. In the first part, we show that in general we do not expect the $C^{1,\alpha_F}$ regularity to hold for a non-homogeneous operator. In the second part, we show that we can relax the 1-homogeneity condition to some algebraic control of the convergence of the operator to its blow-down (in analogy to the control given by H\"older coefficients for linear equations in non-divergence form). 

\subsection{A counterexample} Let us start with the proof of Theorem~\ref{thm.counterexample}, saying that in general solutions are not $C^{1,\alpha_F}$ by constructing such a solution for a particular $F$. 

\begin{proof}[Proof of Theorem~\ref{thm.counterexample}]
Let us construct such a solution.  We will do so in $\R^2$ and for some $F$ that will be, moreover, rotationally invariant. Here, $\alpha_F$ is given by \eqref{eq.alphan_opt}.

Let us consider, without loss of generality, the domain $Q_1 = [-1, 1]\times[-1, 1] \subset \R^2$, and the thin obstacle problem in this domain. In particular the thin space is given by $\{x_2 = 0\}$.

We will construct a sequence of solutions $u_i$ to a thin obstacle problem in $Q_1$ with operator $F_i$ of the form \eqref{eq.F}, all with the same boundary datum given by 
\[
u_i = \pm 1 \quad\text{on}\quad \{x_1 = \pm 1\},\qquad u_i = -1 \quad\text{on}\quad \{x_2 = \pm 1\},\quad\text{for all}\quad i\in \N.
\]

In particular, notice that $\partial_{x_1} u_i$ either satisfies an equation with bounded measurable coefficients, or vanishes. Moreover, on the boundary we have $\partial_{x_1} u_i\ge 0$, hence by the maximum principle $\partial_{x_1} u_i \ge 0$ everywhere. On the other hand, since $F_i$ will be rotationally invariant, $u_i$ is even with respect to $\{x_2 = 0\}$. Thus, if we let $Q^+ := [-1, 1]\times[0, 1]$, then $\partial_2 u_i \le 0$ on $\partial Q^+$ and again by maximum principle $\partial_2 u_i \le 0$ in $Q^+$ (and by symmetry, $\partial_2 u_i \ge 0$ in $Q^-$).

In particular, for each $u_i$ there exists some $z_i\in (-1, 1)$ such that $u_i = 0$ on $\{x_2 = 0, x_1 \le z_i\}$ and $u_i > 0$ on $\{x_2 = 0, x_1 > z_i\}$. That is, $(z_i, 0)$ is (the unique) free boundary point for $u_i$ (and from the monotonicity of the solution around it, it is a regular free boundary point). 

Let us now suppose $F_0(D^2 w) := \Delta w$. We will construct a sequence with $F_i$ having ellipticity constants $1$ and $2$, so that the corresponding $u_i$ will converge to some $u_\infty$ satisfying a fully nonlinear thin obstacle problem, with some operator $F_\infty$ with ellipticity constants $1$ and $2$ (all up to subsequences if necessary). 

Let $\Lambda_i = 2-\frac{1}{i+1}$, and let us define 
\[
F_i(A) = \max_{0 \le j \le i}\{ \mathcal{P}_{1,\Lambda_j}^+(A) - \tilde c_j\} = \max\{F_{i-1}(A), \mathcal{P}_{1,\Lambda_i}^+(A) - \tilde c_i\}
\]
for some sequence $\tilde c_i \ge 0$ to be chosen. In particular, 
\[
\|u_i\|_{C^{1,\alpha(\Lambda_i)}(B_{1/2}\cap \{x_2 = 0\})}\le C_i
\]
for some $C_i$, where we recall $\alpha(\Lambda_i)$ is the optimal homogeneity associated to $\mathcal{P}_{1,\Lambda_j}^+$ (the homogeneity of any blow-up). In particular, from Proposition~\ref{prop.pucci_res}, $\alpha_i := \alpha(\Lambda_i)$ is strictly increasing, converging to $\alpha_\infty:= \alpha(2)\in (0, 1)$ as $i \to \infty$. 

On the other hand, recall that the solution to the fully nonlinear thin obstacle problem can be defined to be the minimal supersolution above the thin obstacle. In particular, since any supersolution with operator $F_i$ above the thin obstacle, is also a supersolution with operator $F_{i-1}$ above the thin obstacle, this implies that $u_i \ge u_{i-1}$ for all $i\in \N$ and the sequence $u_i$ is pointwise non-decreasing. 

On the other hand, we claim (and prove later) that for $\eps_i = \frac{1}{2(i+1)}$, 
\begin{equation}
\label{eq.limsup}
\limsup_{r> 0} \frac{\|u_i\|_{L^\infty((z_i, z_i+r)\times\{0\})}}{r^{1+\alpha_i+\eps_i}} = +\infty. 
\end{equation}

Let us suppose that we have constructed $u_i$ already, let us show how to construct $u_{i+1}$. Observe that, once $F_i$ is fixed and hence $u_i$ is fixed, we only need to fix $\tilde c_{i+1}$ to determine $u_{i+1}$. Observe, also, that as $\tilde c_{i+1}\to \infty$, $u_{i+1}\downarrow u_i$ converges locally uniformly. Not only that, but also from interior estimates $z_{i+1}\uparrow z_i$. 

\begin{figure}
\includegraphics[scale = 1]{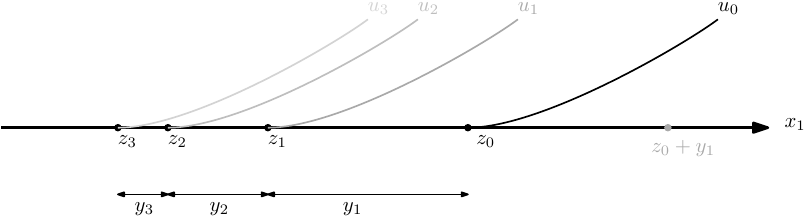}
\caption{The sequence of solutions to fully nonlinear thin obstacle problems, $u_i$, with free boundary points at $z_i$ on the thin space, $x_2 = 0$. }
\label{fig.1}
\end{figure}
Let us denote $y_{i+1}:= z_i - z_{i+1}$, so that $z_i + y_1+y_2+\dots+y_i = z_0 $ (see Figure~\ref{fig.1} for a sketch). Since $y_{i+1}\downarrow 0$ as $\tilde c_{i+1}\uparrow \infty$, in order to determine $\tilde c_{i+1}$ it is enough to say how small $y_{i+1}$ needs to be. 

Notice that from \eqref{eq.limsup}, there exists some sequence $r_j\downarrow 0$ such that 
\[
\frac{u_i(z_i+r_j,0)}{r_j^{1+\alpha_i+\eps_i}} \ge 1.
\]

Let us now fix $\bar r\in (r_j)_{j\in \N}$ small enough such that 
\[
r_j^{1+\alpha_i+\eps_i} \ge  i r_j^{1+\alpha_\infty},
\]
which always exists, since $\alpha_i +\eps_i < \alpha_\infty$. Let us also suppose that $\bar r < \frac{y_i}{2}$, and fix $y_{i+1} = \bar r$. In particular this determines $\tilde c_{i+1}$ and we have constructed $u_{i+1}$ as the solution to the fully nonlinear thin obstacle problem with operator $F_{i+1}$. 

Let us check that $u_\infty\notin C^{1,\alpha_\infty}$, thus giving our counterexample. To do so, let us denote by $(z_\infty, 0)$ the free boundary point of $u_\infty$ (in particular, $z_0 = z_\infty +y_1+y_2+\dots$), so that $z_i\downarrow z_\infty$ as $i\to \infty$. Notice now that
\[
u_\infty(z_i+y_{i+1}, 0) \ge u_i(z_i+y_{i+1}, 0) \ge y_{i+1}^{1+\alpha_i+\eps_i} \ge i y_{i+1}^{1+\alpha_\infty}. 
\]
On the other hand, $z_i - z_\infty = y_{i+1}+y_{i+2}+\dots$. From the recursive condition $y_{i+1}\le y_i/2$ we get $z_i - z_\infty\le 2y_{i+1}$, so that, since $u_\infty$ is monotone in the $\be_1$ direction, we have 
\[
\frac{u_\infty(z_i+y_{i+1}, 0) - u_\infty(z_\infty, 0)}{(z_i +y_{i+1} - z_\infty)^{1+\alpha_\infty}} \ge \frac{1}{3^{1+\alpha_\infty}} \frac{u_\infty(z_i+y_{i+1}, 0)}{y_{i+1}^{1+\alpha_\infty}} \ge \frac{1}{3^{1+\alpha_\infty}} i\to \infty 
\]
as $i\to \infty$. 
This, together with the fact that $\nabla u_\infty(z_\infty)(0) = 0$, yields that $u_\infty\notin C^{1,\alpha_\infty}$, as we wanted to see. 
\\[0.1cm]
{\it Proof of claim, \eqref{eq.limsup}:}  Let us finish by proving the claim \eqref{eq.limsup}. By translating if necessary, let us assume $z_i = 0$.

Recall that, by uniqueness of blow-ups (Proposition~\ref{prop.uniqueness} or Corollary~\ref{cor.mainprop}), we have that 
\[
\frac{u_i(rx)}{\|u_i\|_{L^\infty(B_r)}} \to \tilde u_i(x)
\]
as $r\downarrow 0$, locally uniformly in $B_1$, where $\tilde u_i(x)$ is the blow-up solution: it satisfies $\|u_i\|_{L^\infty(B_1)} = 1$ and it is $1+\alpha_i$ homogeneous. In particular, given for any $\delta > 0$ there exists some $r_0 > 0$ such that 
\[
\left\|\frac{u_i(rx)}{\|u_i\|_{L^\infty(B_r)}} - \tilde u_i(x)\right\|_{L^\infty(B_{1/2})} \le \delta\qquad\text{for all}\quad r < r_0. 
\]
By triangular inequality, and recalling that $\tilde u_i(x)$ is $1+\alpha_i$ homogeneous, we have that 
\[
\frac{\|u_i\|_{L^\infty(B_{r/2})}} {\|u_i\|_{L^\infty(B_r)}} \ge 2^{-1-\alpha_i} -\delta \ge 2^{-1-\alpha_i-\eps_i/2}
\]
if $\delta> 0$ is small enough depending only on $\alpha_i$ and $\eps_i$. 
Now, since $\|u_i\|_{L^\infty(B_1)} = 1$, and applying it iteratively, we have that 
\[
\|u_i\|_{L^\infty(B_{2^{-k} })}\ge 2^{-(1+\alpha_i+\eps_i/2)k}.
\]
In particular, \eqref{eq.limsup} holds. 
%
%
\end{proof}

\subsection{Optimal regularity for 
operators  with quantified convergence to the recession function}

Suppose, now, that the operator $F$ we are dealing with, satisfying \eqref{eq.F}, is not 1-homogeneous, but we have some control on its convergence towards the recession function $F^*$. In particular, let us define the monotone modulus of continuity $\omega_F:[0, 1]\to [0, C\Lambda]$ given by\footnote{This is equivalent to asking that, up to constants, $d(F_\tau, F^*) \le \omega_F(\tau)$ where $F_\tau := \tau F(\tau^{-1}\cdot)$ and $d(F, G) := \sup_{M\in \mathcal{M}_n^S} \frac{|F(M) - G(M)|}{1+\|M\|}$. }
\begin{equation}
\label{eq.conditionkappa0}
\omega_F(\tau) :=  \sup_{0\le \xi \le \tau}\sup_{\|A\|\le 1}  |\xi F(\xi^{-1} A) - F^*(A)|,
\end{equation}
so that, in particular, 
\[
|F(M) - F^*(M)| \le \|M\|\omega_F(\|M\|^{-1})
\]
for all $M\in \mathcal{M}_n^S$ with $\|M\|\ge 1$. Observe that $\omega_F$ is indeed a modulus of continuity, by the definition of $F^*$.

We are now interested in the case where there exists some power $\kappa > 0$ such that
\begin{equation}
\label{eq.conditionkappa}
\omega_F(\tau) \le  C_\kappa\tau^{\kappa}\quad\text{for all}\quad\tau\in (0, 1)
\end{equation}
for some $C_\kappa > 0$ fixed. 

We have the analogous results in this case. We start with the analogue to Proposition~\ref{prop.opt_reg}, where the proof follows exactly as before and we highlight the main differences: 

\begin{prop}
\label{prop.opt_reg_cc}
Proposition~\ref{prop.opt_reg} also holds when $F$ satisfies \eqref{eq.F} and \eqref{eq.conditionkappa0}-\eqref{eq.conditionkappa} for some $\kappa > 0$ and $C_\kappa > 0$. 
%
%
%
\end{prop}
\begin{proof}
The proof follows along the lines of the proof of Proposition~\ref{prop.opt_reg}, and so we use the same notation as in  there; the dependence on $F$ will now also include a dependence on $\kappa$ and $C_\kappa$ from \eqref{eq.conditionkappa0}-\eqref{eq.conditionkappa}. The only difference is that the functions in the sequence $u_{k_m}(r_{m} x)$ satisfy a different equation outside of the thin space (still, they are solutions to a fully nonlinear thin obstacle problem), and consequently the expression \eqref{eq.equation} becomes, in this case, 
\begin{equation}
\label{eq.equationG}
\mathcal{P}_{\lambda, \Lambda}^- (D^2 (v_m - \beta u_0^{(k_m)}))  \le G_{m, \beta}(x) \le \mathcal{P}_{\lambda, \Lambda}^+ (D^2 (v_m - \beta u_0^{(k_m)}))  
\end{equation}
for some function $G_{m, \beta}(x)$. We remark that $u_0^{(k_m)}$ is a solution to the fully nonlinear thin obstacle problem with operator $F^*$, which is different from $F$. 

Let us find an expression and a bound for $G_{m,\beta}(x)$. We know that 
\[
\begin{split}
\mathcal{P}_{\lambda, \Lambda}^- (D^2 (v_m - \beta u_0^{(k_m)}))  \le&  F_m\left(\frac{r_m^2}{q_m} (D^2 u_{k_m})(r_m x)\right) \\
& - F_m\left((Q_{k_m}(r_m)q_m^{-1}+\beta)r_m^2 (D^2 u_{0}^{(k_m)})(r_m x)\right),
\end{split}
\]
where 
\[
F_m(A) := \frac{r_m^2}{q_m} F\left(\frac{q_m}{r_m^2} A\right),\qquad q_m := \|u_{k_m} - \phi_m\|_{L^\infty(B_{r_m})},
\]
so that, in particular, 
\[
\mathcal{P}_{\lambda, \Lambda}^- (D^2 (v_m - \beta u_0^{(k_m)}))  \le - \frac{r_m^2}{q_m} F\left((Q_{k_m}(r_m)+q_m \beta)(D^2 u_{0}^{(k_m)})(r_m x)\right).
\]
The analogous inequality holds for $\mathcal{P}_{\lambda, \Lambda}^+$. Thus, in \eqref{eq.equationG} we define 
\[
G_{m, \beta}(x) := - \frac{r_m^2}{q_m} F\left((Q_{k_m}(r_m)+q_m \beta)(D^2 u_{0}^{(k_m)})(r_m x)\right),
\]
and let us now find bounds for $G_{m,\beta}(x)$. Observe that, by homogeneity, and after a rotation assuming that $\nu_{k_m} = \be_1$,
\[
|D^2 u_0^{(k_m)}(r_m x)| = C r_m^{\alpha_{k_m}-1}|(x_1, x_n)|^{\alpha_{k_m} - 1}.
\]
Therefore, from \eqref{eq.conditionkappa}, using that $F^*(D^2 u_0^{(k_m)}) = 0$ and $F^*$ is 1-homogeneous, 
\[
|G_{m, \beta}(x)| \le \frac{C(Q_{k_m}(r_m)+q_m\beta)}{r_m^{\Pi - 1- \alpha_{k_m}}\theta_m|(x_1, x_n)|^{1-\alpha_{k_m}}}\omega_F\left(\frac{r_m^{1-\alpha_{k_m}} |(x_1, x_n)|^{1-\alpha_{k_m}}}{C(Q_{k_m}(r_m)+q_m\beta)}\right).
\]
Now, from \eqref{eq.qbound} and the monotony of $\omega_F$, we have 
\begin{equation}
\label{eq.boundG}
\begin{split}
|G_{m,\beta}(x) | & \le \frac{C(1+\beta r_m^{\Pi})}{r_m^{\Pi - 1- \alpha_{k_m}}|(x_1, x_n)|^{1-\alpha_{k_m}}} \omega_F(C r_m^{1-\alpha_{k_m}})\\
& \le C_\kappa C r_m^{\kappa (1-\alpha_{k_m})-\Pi +1+\alpha_{k_m}}\frac{1+\beta r_m^{\Pi}}{|(x_1, x_n)|^{1-\alpha_{k_m}}},
\end{split}
\end{equation}
where we are also using that $Q_{k_m}(r_m)\to  \infty$. We choose $\sigma$ small enough (depending on $F$ and $\kappa$) so that $\kappa (1-\alpha_{k_m})-\Pi +1+\alpha_{k_m} > 0$, and \eqref{eq.boundG} goes to zero as $m\to \infty$. In particular, observe that $G_{m,\beta}$ is locally bounded in the domain where it is defined in \eqref{eq.equationG}, so that the expression makes sense, and it belongs to $L^p_{\rm loc}$ for some $p > 2$. In particular, we have obtained an analogy to \eqref{eq.equation}. Combined with barriers from above and below (similarly to the ones constructed in \cite[Lemma 7.2]{RS17}) together with the interior H\"older regularity for equations in non-divergence form with $L^p$ right-hand side (for $p \ge n$, see \cite[Proposition 4.10]{CC95}), we obtain by analogy with Step 4 of Proposition~\ref{prop.opt_reg} the uniform convergence of the blow-up sequence. 

On the other hand, as in Step 5 of the proof of Proposition~\ref{prop.opt_reg}, $v_m$ converges to a two-dimensional functional (see \eqref{eq.unif_vm}).

Let us fix some $m\in \N$, and assume $\nu_{k_m} = \be_1$.  Notice that $v_m$ is locally $C^{2,\alpha}$ outside of $U_m^+ = \Omega_{k_m}'\cup\{x_n = 0, x_1 \le 0\}$, in particular, it satisfies an equation with bounded measurable coefficients  there:
\[
\sum_{i, j = 1}^n a^m_{ij}(x) \partial_{ij} v_m(x) = G_{m, 0}(x) \quad\text{in}\quad B_1\setminus U_m^+,
\]
where $A_m(x) = (a^m_{ij}(x))_{i,j = 1}^n \in \mathcal{M}_n^S$ is uniformly elliptic with ellipticity constants $\lambda$ and $\Lambda$, and $G_{m, 0}$ has bounds given by \eqref{eq.boundG}.

Since we have that for some $p > 2$
\[
G_{m,\beta}(x_1,0,\dots,0, x_n) \in L^p(B_1),\quad\text{with}\quad B_1\subset \R^2
\]
uniformly in $m$, arguing as in Step 6 of the proof of Proposition~\ref{prop.opt_reg} we get that $v_m$ (and $v_m - \beta u_0^{(k_m)}$) converge on the two-dimensional slice $(x_1, 0,\dots, 0, x_n)$. By also using the fact that $\|G_{m, \beta}(x_1,0,\dots,0, x_n)\|_{L^p(B_1)} \to 0$ as $m\to \infty$, for $B_1\subset \R^2$, we finish the proof as the one in Proposition~\ref{prop.opt_reg}.
%
%
\end{proof}

\begin{rem}
\label{rem.fin}
Once we have Proposition~\ref{prop.opt_reg_cc}, then Theorem~\ref{thm.expansion} also holds in this case. 

Furthermore, Lemma~\ref{lem.second} and Theorem~\ref{thm.opt5}  also hold for $F$ satisfying \eqref{eq.conditionkappa0}-\eqref{eq.conditionkappa}, using Proposition~\ref{prop.opt_reg_cc} instead of Proposition~\ref{prop.opt_reg}. Observe that Lemma~\ref{lem.first} is already stated for general operators, so the results follow in the same way. 
\end{rem}

We can now prove the analogue to Theorem~\ref{thm.main_1homog} in this case:

\begin{thm}
\label{thm.main_1homog_2}
Theorem~\ref{thm.main_1homog} also holds when $F$ satisfies \eqref{eq.F} and \eqref{eq.conditionkappa0}-\eqref{eq.conditionkappa} for some $\kappa > 0$ and $C_\kappa > 0$. 
%
%
%
%
\end{thm}
\begin{proof}
The proof follows exactly the same as Theorem~\ref{thm.main_1homog}, thanks to the observation in Remark~\ref{rem.fin}.
\end{proof}

\section{Appendix: the boundary Harnack}
\label{sec.appendix}
The following is the boundary Harnack comparison principle but for functions that do not necessarily satisfy an equation with the same operator. 

We denote by $B_1^+ = B_1\cap \{x_1 \ge 0\}$, and we consider operators of the form 
\begin{equation}
\label{eq.op_div_form}
\mathcal{L} w (x) := {\rm div}(A(x) \nabla w(x)), \quad 0 < \lambda{\rm Id}\le A(x) \le \Lambda{\rm Id},
\end{equation}
that is, uniformly elliptic operators in divergence form with bounded measurable coefficients.

\begin{thm}[Boundary Harnack]
\label{thm.boundaryHarnack}
Let $0<\lambda_\circ\le \Lambda_\circ$ and let $u, v \in C(\overline{B_1^+})\cap H^1(\overline{B_1^+})$ such that $v\ge 0$ and they satisfy 
\begin{equation}
\left\{
\begin{array}{ll}
\mathcal{L}_{\alpha\beta} (\alpha u - \beta v)  =  0&\quad \text{in}\quad B_1^+\\
u = v = 0&\quad\text{on}\quad B_1\cap \{x_1 = 0\}\\
u(\frac12 \be_1)\le 1, v(\frac12 \be_1) \ge 1,
\end{array}
\right.
\end{equation}
in the weak sense, for any $\alpha, \beta\ge 0$, and for some elliptic operator in divergence form with bounded measurable coefficients, $\mathcal{L}_{\alpha\beta}$ as in \eqref{eq.op_div_form}, with ellipticity constants $\lambda_{\alpha\beta}$ and $\Lambda_{\alpha\beta}$ such that $0< \lambda_\circ \le \lambda_{\alpha\beta}\le \Lambda_{\alpha\beta}\le \Lambda_\circ< \infty$.

Then, 
\[
c_* u \le v \quad\text{in}\quad B_{1/2}^+
\]
for some constant $c_*$ depending only on $n$, $\lambda_\circ$, and $\Lambda_\circ$. 

In particular, if $u \ge 0$ and $u(\frac12\be_1) = v(\frac12\be_1) = 1$, 
\begin{equation}
\label{eq.boundboth}
c_* \le \frac{u}{v} \le \frac{1}{c_*}\quad\text{in}\quad B_{1/2}^+
\end{equation}
for some constant $c_*$ depending only on $n$, $\lambda_\circ$, and $\Lambda_\circ$. 
\end{thm}

In order to show Theorem~\ref{thm.boundaryHarnack}, we will follow the boundary Harnack proof of De Silva and Savin \cite{DS20}, and more precisely, we will follow its adaptations in \cite[Appendix B]{FR21} and \cite{RT21}. Let us start by stating the two following lemmas.

In the first one, we have a global bound on the $L^\infty$ norm of the solution, up to the boundary. It corresponds to \cite[Lemma 3.7]{RT21} (see also \cite[Lemma B.6]{FR21}). 

\begin{lem}[\cite{RT21}]
\label{lem.boundC}
Let $u\in C(\overline{B_1^+})$ such that $\mathcal{L} u = 0$ in $B_1^+$ for an operator of the form \eqref{eq.op_div_form} with ellipticity constants $\lambda$ and $\Lambda$, and $u = 0$ on $\{x_1 = 0\}$. Assume, moreover, that $u(\frac12 \be_1) = 1$. Then 
\[
\|u \|_{L^\infty(B_{1/2})}\le C 
\]
for some $C$ depending only on $n$, $\lambda$, and $\Lambda$. 
\end{lem}

In the second one, we have that if the solution is very positive away from the boundary, and not too negative near the boundary, it is in fact positive up to the boundary. 

\begin{lem}
\label{lem.delta}
Let $\mathcal{L}$ be of the form \eqref{eq.op_div_form} with ellipticity constants $\lambda$ and $\Lambda$. Then, there exists a $\delta > 0$ depending only on $n$, $\lambda$, and $\Lambda$, such that if $w\in C(\overline{B_1^+})$ satisfies
\[
\left\{
\begin{array}{ll}
\mathcal{L} w  =  0&\quad \text{in}\quad B_1^+\\
w  = 0&\quad\text{on}\quad B_1\cap \{x_1 = 0\}
\end{array}
\right.
\quad\text{and}\quad 
\left\{
\begin{array}{ll}
 w  \ge  1 & \quad \text{in}\quad B_1^+ \cap\{x_1\ge \delta\}\\
w  \ge -\delta &\quad\text{in}\quad B_1^+
\end{array}
\right.,
\]
in the weak sense, then $w\ge 0$ in $B_{1/2}^+$. 
\end{lem}
\begin{proof}
The proof is identical to the proof of \cite[Proposition 3.10]{RT21}, where the modulus of continuity of the operator is only used to deal with the right-hand side (which here is just zero). Alternatively, we refer the reader to the proof of \cite[Proposition B.8]{FR21}. 
\end{proof}

We can now give the proof of Theorem~\ref{thm.boundaryHarnack}.

\begin{proof}[Proof of Theorem~\ref{thm.boundaryHarnack}]
We proceed as in the proof of \cite[Theorem 1.1]{DS20} (see also the proof of \cite[Theorem B.1]{FR21}). 

By Lemma~\ref{lem.boundC} we can assume that  $u$ is bounded in $B_{3/4}^+$ by some constant $C$ depending only on $n$, $\lambda_\circ$ and $\Lambda_\circ$. We consider 
\[
w_{M\eps} = M v - \eps u,
\]
for some constants $M>0$ (large) and $\eps>0$ (small) to be chosen. Observe that, since $u$ is bounded and $v \ge 0$, $w_{M\eps} \ge -\eps C$ in $B_1^+$. In particular, given the $\delta$ from Lemma~\ref{lem.delta} with ellipticity constants $\lambda_\circ$ and $\Lambda_\circ$, we can choose $\eps$ small enough (depending only on $n$, $\lambda_\circ$, and $\Lambda_\circ$) such that 
\[
w_{M\eps} \ge -\delta\quad\text{in}\quad B_1^+. 
\]

On the other hand, by the interior Harnack inequality for operators in divergence form (\cite[Theorem 8.20]{GT83}) we can choose $M$ large enough (depending only on $n$, $\lambda_\circ$, and $\Lambda_\circ$) such that $M v \ge 1+\delta$ in $B_{3/4}\cap \{x_1\ge\delta\}$. 

Thus, 
\[
w_{M\eps} \ge 1\quad\text{in}\quad B_{3/4}\cap \{x_1\ge \delta\}. 
\]

Since $w_{M\eps} = 0$ on $\{x_1 = 0\}$ and $\mathcal{L}_{M\eps} w_{M\eps} = 0$ in $B_1^+$ by assumption, we can apply Lemma~\ref{lem.delta} to deduce that $w_{M\eps} \ge 0$ in $B_{1/2}^+$ (up to a covering argument and taking $\delta$ smaller if necessary). 

That is, we deduce 
\[
v \ge \frac{\eps}{M} u\quad\text{in}\quad B_{1/2}^+.
\]
Since $M$ and $\eps$ depend only on $n$, $\lambda_\circ$, and $\Lambda_\circ$ we deduce one inequality in the desired result. 

Finally, if $u \ge 0$ and $u(\frac12\be_1) = v(\frac12\be_1) = 1$, by exchanging the roles of $u$ and $v$ we get \eqref{eq.boundboth}. 
\end{proof}

\end{document}